\documentclass[a4paper]{amsart}
\usepackage[margin=1in]{geometry}

\usepackage{amsmath,amssymb,amsthm,latexsym,mathtools,thmtools,amsfonts,
enumerate,wasysym,nccmath}
\usepackage[T1]{fontenc}
\usepackage{tikz}
\usetikzlibrary{arrows,matrix}
\usepackage[english]{babel}
\usepackage[parfill]{parskip}
\usepackage{cancel}
\usepackage{colortbl}
\usepackage{makecell}

\numberwithin{equation}{section}

\theoremstyle{plain}
\newtheorem{theorem}{Theorem}[section]
\newtheorem*{theorem*}{Theorem}
\newtheorem{proposition}[theorem]{Proposition} 
\newtheorem*{proposition*}{Proposition}
\newtheorem{corollary}[theorem]{Corollary}
\newtheorem*{corollary*}{Corollary}
\newtheorem{lemma}[theorem]{Lemma}     
\newtheorem{definition}[theorem]{Definition}          
\newtheorem{conjecture}[theorem]{Conjecture}
\newtheorem*{conjecture*}{Conjecture}

\newtheorem{claim}[theorem]{Claim}

\theoremstyle{remark}
\newtheorem*{notation}{Notation}
\newtheorem{remark}[theorem]{Remark}
\newtheorem{example}[theorem]{Example}

\usepackage[colorlinks]{hyperref}
\hypersetup{
    colorlinks,
    linkcolor={red!50!black},
    citecolor={green!50!black},
    urlcolor={blue!80!black}
}

\newcommand{\calA}{\mathcal{A}}
\newcommand{\calB}{\mathcal{B}}
\newcommand{\calF}{\mathcal{F}}
\newcommand{\calS}{\mathcal{S}}

\newcommand{\bbC}{\mathbb{C}}
\newcommand{\bbF}{\mathbb{F}}
\newcommand{\bbG}{\mathbb{G}}
\newcommand{\bbN}{\mathbb{N}}
\newcommand{\bbP}{\mathbb{P}}
\newcommand{\bbR}{\mathbb{R}}
\newcommand{\bbZ}{\mathbb{Z}}

\newcommand{\str}{\mathrm{str}}
\newcommand{\slrk}{\mathrm{sl.rk}}

\newcommand{\HS}{\mathrm{HS}}
\renewcommand{\HF}{\mathrm{HF}}

\newcommand{\wFC}{\mathrm{wFC}}
\newcommand{\sFC}{\mathrm{sFC}}
\DeclareMathOperator{\codim}{codim}
\DeclareMathOperator{\coeff}{coeff}

\newcommand{\exampleend}{\,\hfill$\spadesuit$}
\newcommand{\remarkend}{\,\hfill$\clubsuit$}


\begin{document}
\title{{On the strength of general polynomials}}

\author{Arthur Bik}
\address[Arthur Bik]{Mathematisches Institut -- U. Bern, Alpeneggstrasse 22 -  CH-3012 Bern, Switzerland}
\email{arthur.bik@math.unibe.ch}

\author{Alessandro Oneto}
\address[Alessandro Oneto]{Dip. di Matematica -- U. di Trento, Via Sommarive, 14 - 38123 Povo (Trento), Italy}
\email{alessandro.oneto@unitn.it}

\begin{abstract}
A slice decomposition is an expression of a homogeneous polynomial as a sum of forms with a linear factor. A strength decomposition is an expression of a homogeneous polynomial as a sum of reducible forms. The slice rank and strength of a polynomial are the minimal lengths of such decompositions, respectively. The slice rank is an upper bound for the strength and the gap between these two values can be arbitrary large. However, in line with a conjecture by Catalisano et al. on the dimensions of secant varieties of the varieties of reducible forms, we conjecture that equality holds for general forms. By using a weaker version of Fr\"oberg's Conjecture on the Hilbert series of ideals generated by general forms, we show that our conjecture holds up to degree $7$ and in degree $9$. 
\end{abstract}
\maketitle

\section{Introduction and main results}

Additive decompositions of homogeneous polynomials can be a tool to provide useful classifications. We consider the ring $\calS = \Bbbk[x_0,\ldots,x_n]$ of polynomials in $n+1$ variables with coefficients in an algebraically closed field $\Bbbk$, equipped with the standard gradation $\calS = \bigoplus_{d\geq 0}\calS_d$, where $\calS_d$ denotes the $\Bbbk$-vector space of degree-$d$ homogeneous polynomials, or forms. Let $f\in\calS_d$ be a form of degree $d\geq 2$.

First, we consider \textit{(symmetric) slice decompositions} of $f$; that is, expressions of the form
\begin{equation}\label{eq:sliceDecomposition}
f = \ell_1g_1+\ldots+\ell_rg_r, 
\end{equation}
where the $\ell_i$ are linear forms. We call the smallest length of such a decomposition the \textit{(symmetric) slice rank} of $f$. We denote it by $\slrk(f)$.

From a slice decomposition such as \eqref{eq:sliceDecomposition}, it is clear that the hypersurface defined by the vanishing of~$f$ contains a linear space of codimension $r$. A classical example of this type of decomposition and their relation with the geometry of linear spaces on hypersurfaces goes back to the work of Cayley and Salmon: the properties of the celebrated $27$ lines lying on a smooth cubic surface are related to the $120$ ways to write the corresponding quaternary cubic as $f = \ell_1\ell_2\ell_3 + m_1m_2m_3$, where the $\ell_i$'s and $m_i$'s are linear forms. See \cite{HLV19:CayleySalmon} for a recent exposition of these equations and the related literature.

The term \textit{slice decomposition} appeared in \cite{ST16:blog} in the context of ordinary tensors: these decompositions have been used to study subsets of $\bbF^n_q$ with no three-terms arithmetic progressions  \cite{B+17:CapSets,CLP17:Progression,EG17:LargeSubsets}. In both settings, slice decompositions consist of sums of terms with a linear factor. However, even when we view homogeneous polynomials as symmetric tensors, the symmetric slice rank we consider here is different than the slice rank defined in \cite{ST16:blog}; see Remark \ref{rmk:differencesSliceRk}.

The exact value of the slice rank for a general form of degree $d$ in $n+1$ variables is known to be
\begin{equation}\label{eq:slrk_formula}
\slrk_{d,n}^\circ := \min\left\{r \in \bbZ_{\geq(n+1)/2} \,\middle|\, r(n+1-r) \geq \binom{d + n - r}{d}\right\};
\end{equation}
see Corollary \ref{cor:maximalSlice}. It equals the smallest codimension of a linear space contained in the general hypersurface of degree $d$ in $\bbP^n$, i.e., the smallest $r$ such that the \textit{Fano variety} of linear spaces of codimension~$r$ in the general hypersurface of degree $d$ in $n$-dimensional projective space is nonempty. The dimension of such Fano varieties (and therefore their non-emptiness) is well-known (see Section \ref{sec:Fano_slicerank}). We refer for example to \cite[Example~12.5]{Harris:First} or the recent survey \cite[Section 2]{CZ19:LinesConics}. As observed in \cite{CCG08:CIinHypersurfaces}, a more algebraic approach involves the study of the dimension of the $k$-th secant variety of the variety of forms with a linear factor in the projective space of degree-$d$ forms: this dimension can be computed by using a result by Hochster and Laksov \cite[Theorem 1]{HL87:LinearSyzygies} showing that an ideal generated by general forms of the same degree do not have linear syzygies.

Slice decompositions are a special case of \textit{strength decompositions}; that is, expressions of the form
\begin{equation}\label{eq:strengthDecomposition}
	f = g_1h_1 + \ldots + g_rh_r, 
\end{equation}
where $\deg(g_i), \deg(h_i) > 0$. The smallest length of such a decomposition is called the \textit{strength} of $f$. We denote it by $\str(f)$. Note that our notion of strength differs from the original one in \cite{AH:Stillman} by $1$. The advantage of the definition we use is that, like for other notions of rank, a form has strength $\leq r$ if and only if it can be written as a sum of $r$ forms of strength $\leq 1$.

From a decomposition such as \eqref{eq:strengthDecomposition}, it is clear that the variety defined by the vanishing of the $g_i$ is contained in the hypersurface defined by the vanishing of $f$, but there is no reason to expect that the $g_i$ form a complete intersection. However, this can be assumed for the general hypersurface; see~\cite{CCG08:CIinHypersurfaces}.
Strength decompositions were used by Ananyan and Hochster in \cite{AH:Stillman} to prove a famous conjecture by Stillman on the existence of a uniform upper bound, independent on the number of variables, on the projective dimension of a homogeneous ideals in polynomial rings.

Since then, the notion of strength has been prominent in several works: Ananyan and Hochster used it to study explicit Stillman bounds \cite{AH:StillmanBounds}; Erman, Sam and Snowden used it in their works on big polynomial rings, also in connection with Hartshorne's conjecture \cite{ESS:Hartshorne}; Bik, Draisma and Eggermont proved the universality of the notion of strength in \cite{BDE:strength}, generalizing previous results of Kazhdan and Ziegler \cite{KZ:rankspoly} and of Derksen, Eggermont and Snowden \cite{DES17:TopologicalCubic}. 
%%In particular, they ssandro{showed} that a form admits a \textit{short} expression of the form \eqref{eq:strengthDecomposition} if the singular locus of the hypersurface defined by the vanishing of $f$ have \textit{small} codimension. 
Moreover, Ballico and Ventura generalized the notion of strength and symmetric slice rank to sections of line bundles over algebraic varieties \cite{BV:lineBundles}.

However, the knowledge on strength of polynomials is still very limited. For example, on the space of homogeneous polynomials of fixed degree $d$ and fixed number of variables $n+1$, neither the general nor the maximal value of the strength is known for arbitrary $n,d$. We have
\begin{center}
	general strength $\leq$ maximal strength $\leq$ maximal slice rank $=$ general slice rank,
\end{center}
where the latter equality follows from the fact that the property of having bounded slice rank is a Zariski-closed condition (see Section \ref{sec:Fano_slicerank}). In this paper, we want to address the following conjecture which states that each of these values are equal. 
\begin{conjecture}\label{conj:genStr=genSl}
Let $f$ be a general form of degree $d\geq2$ in $n+1$ variables. Then $\str(f) = \slrk_{n,d}^\circ$.
\end{conjecture}
As far as we know, this conjecture has not been explicitly stated before in the literature; however, it was implicitly given within the analysis in \cite{C+19:SecantReducible} on the dimension of secant varieties of the variety of reducible forms. In particular, Conjecture \ref{conj:genStr=genSl} is implied by the following stronger conjecture. 

\begin{conjecture}[{\cite[Remark 7.7]{C+19:SecantReducible}}]\label{conj:dimEquality_intro}
For all integers $d\geq2$ and $n,r\geq1$, the dimension of the $r$-th secant variety of the variety of reducible forms of degree $d$ in $n+1$ variables is equal to the dimension of the $r$-th secant variety of the subvariety of forms with a linear factor. 
\end{conjecture}

\begin{remark}
In fact, we conjecture that the $r$-th secant variety of the subvariety of forms with a linear factor is the unique component of the $r$-th secant variety of the variety of reducible forms of degree $d$ in $n+1$ variables unless $(n,d,r)=(3,4,2)$. See Remark~\ref{rmk:one-maximal}.
\remarkend
\end{remark}

Conjecture \ref{conj:dimEquality_intro} implies Conjecture \ref{conj:genStr=genSl}. Indeed, a general element of the $r$-th secant variety of the variety of reducible forms (respectively forms with a linear factor) has by definition a length-$r$ strength decomposition (respectively slice decomposition) and the conjecture implies that the general form lies on the $r$-th secant variety of the variety of reducible forms if and only if it lies on the $r$-th secant variety of the variety of forms with a linear factor. Conjecture \ref{conj:genStr=genSl}  is known to be true in some cases:\\
(1) If $d \geq \frac{3}{2}n+\frac{1}{2}$ by \cite[Corollary A]{Sz96:CompleteIntersections} where the author studies complete intersection curves on general hypersurfaces. Indeed, in this range, the general slice rank is at least $n-1$.\\
(2) If $2\cdot\slrk_{n,d}^\circ \leq n+2$ by \cite[Theorem 5.1]{CCG08:CIinHypersurfaces}. Here Conjecture \ref{conj:dimEquality_intro} is deduced from the known cases of Fr\"oberg's conjecture on the Hilbert series of ideals generated by general forms. 

Our main results are the following theorem and its corollary.

\begin{theorem}\label{thm:main}
Conjecture \ref{conj:dimEquality_intro} holds for $d \leq 7$ and $d=9$. 
\end{theorem}

\begin{corollary}\label{cor:main}
Conjecture \ref{conj:genStr=genSl} holds for $d \leq 7$ and $d=9$. 
\end{corollary}

As observed in \cite{CCG08:CIinHypersurfaces}, the latter result can be rephrased as follows.

\begin{corollary}
Let $d\in\{2,3,4,5,6,7,9\}$ and $n\geq 1$ be any integers. Then the general hypersurface in~$\bbP^n_\Bbbk$ of degree $d$ does not contain any complete intersection of codimension $r < \slrk_{n,d}^\circ$.
\end{corollary}

See Table~\ref{table:stateoftheart} for an overview of the currently known cases.

Our approach to computing dimensions of the different components of secant varieties of the varieties of reducible forms follows the one of \cite{CCG08:CIinHypersurfaces} and \cite{C+19:SecantReducible}: the tangent spaces at general points are defined by homogeneous parts of ideals generated by general forms and, therefore, we study Hilbert functions of general ideals. However, instead of Fr\"oberg's conjecture (Conjecture \ref{conj:Froberg}) we focus on the inequality (\ref{eq:propertyP}) which implies a weaker version (Conjecture \ref{conj:wFroberg}) where the equality between the Hilbert series of an ideal generated by general forms and the prescribed power series is replaced by a coefficient-wise inequality. We show that this inequality holds in low degrees ($d \leq 7$ and $d = 9$) for the ideals of our interest; see Section \ref{sec:Froberg}. Then, after a careful and technical study of the asymptotic behaviour of the prescribed power series, we deduce that in this numerical range Conjecture \ref{conj:dimEquality_intro} and, consequently, Conjecture \ref{conj:genStr=genSl} hold. Moreover, our analysis allows us to deduce that, for $d \leq 7$ and $d = 9$, the $r$-th secant variety of the variety of forms with a linear factor is the unique maximal-dimension component of the $r$-th secant variety of the variety of reducible forms, except in the case $(n,d,r) = (3,4,2)$ where the components are all of codimension one; see Remark \ref{rmk:one-maximal}. 

\begin{center}
\begin{table}[h]\label{table:stateoftheart}
\begin{tabular}{c | c c c c c c c c c c c c}
	\diaghead(-2,1){aaaaa}%
{$n$}{$d$} & 2 & 3 & 4 & 5 & 6 & 7 & 8 & 9 & 10 & 11 & 12 & $\cdots$ \\
\hline
 2 & \cellcolor{red!50!white} &  \cellcolor{red!50!white} &  \cellcolor{red!50!white} &  \cellcolor{red!50!white} &  \cellcolor{red!50!white} &  \cellcolor{red!50!white} &  \cellcolor{red!50!white} &  \cellcolor{red!50!white} &  \cellcolor{red!50!white} &  \cellcolor{red!50!white} &  \cellcolor{red!50!white} &  \cellcolor{red!50!white} \\
 3 & \cellcolor{green!50!white} &  \cellcolor{green!50!white} &  \cellcolor{red!50!white} &  \cellcolor{red!50!white} &  \cellcolor{red!50!white} &  \cellcolor{red!50!white} &  \cellcolor{red!50!white} &  \cellcolor{red!50!white} &  \cellcolor{red!50!white} &  \cellcolor{red!50!white} &  \cellcolor{red!50!white} &  \cellcolor{red!50!white} \\
 4 & \cellcolor{green!50!white} &  \cellcolor{green!50!white} &  \cellcolor{green!50!white} &  \cellcolor{green!50!white} &  \cellcolor{red!50!white} &  \cellcolor{red!50!white} &  \cellcolor{red!50!white} &  \cellcolor{red!50!white} &  \cellcolor{red!50!white} &  \cellcolor{red!50!white} &  \cellcolor{red!50!white} &  \cellcolor{red!50!white} \\
 5 & \cellcolor{green!50!white} &  \cellcolor{blue!50!white} &  \cellcolor{blue!50!white} &  \cellcolor{blue!50!white} &  \cellcolor{blue!50!white} &  \cellcolor{red!50!white} &  \cellcolor{red!50!white} &  \cellcolor{red!50!white} &  \cellcolor{red!50!white} &  \cellcolor{red!50!white} &  \cellcolor{red!50!white} &  \cellcolor{red!50!white} \\
 6 & \cellcolor{green!50!white} &  \cellcolor{green!50!white} &  \cellcolor{blue!50!white} &  \cellcolor{blue!50!white} &  \cellcolor{blue!50!white} &  \cellcolor{blue!50!white} &  &  \cellcolor{red!50!white} &  \cellcolor{red!50!white} &  \cellcolor{red!50!white} &  \cellcolor{red!50!white} &  \cellcolor{red!50!white} \\
 7 & \cellcolor{green!50!white} &  \cellcolor{blue!50!white} &  \cellcolor{blue!50!white} &  \cellcolor{blue!50!white} &  \cellcolor{blue!50!white} &  \cellcolor{blue!50!white} &  &  \cellcolor{blue!50!white} &  \cellcolor{red!50!white} &  \cellcolor{red!50!white} &  \cellcolor{red!50!white} &  \cellcolor{red!50!white} \\
 8 & \cellcolor{green!50!white} &  \cellcolor{green!50!white} &  \cellcolor{blue!50!white} &  \cellcolor{blue!50!white} &  \cellcolor{blue!50!white} &  \cellcolor{blue!50!white} &  &  \cellcolor{blue!50!white} & & &  \cellcolor{red!50!white} &  \cellcolor{red!50!white} \\
 \vdots & \cellcolor{green!50!white} &  \cellcolor{blue!50!white} &  \cellcolor{blue!50!white} &  \cellcolor{blue!50!white} &  \cellcolor{blue!50!white} &  \cellcolor{blue!50!white} &  &  \cellcolor{blue!50!white} & & & &  \cellcolor{red!50!white} $d \geq \frac{3}{2}n + \frac{1}{2}$ 
\end{tabular}
\begin{tabular}{r l}
	${\color{red!50!white}\blacksquare}$ & : \cite[Corollary A]{Sz96:CompleteIntersections} \\
${\color{green!50!white}\blacksquare}$ & : \cite[Theorem 5.1]{CCG08:CIinHypersurfaces} \\
${\color{blue!50!white}\blacksquare}$ & : Corollary \ref{cor:main}
\end{tabular}

\medskip
\caption{State-of-the-art of Conjecture \ref{conj:genStr=genSl}.} 
\end{table}
\end{center}

\subsection*{Structure of the paper}
Sections~\ref{sec:sliceRanks} and~\ref{sec:strength_geometry}, go over the basic properties of the slice rank and strength of forms, respectively. In Sections~\ref{sec:Froberg}, we discuss Fr\"oberg's conjecture. In Section~\ref{sec:proofMain}, we prove Theorem~\ref{thm:main} assuming Lemma~\ref{lemma:arithmeticLemma}. And finally, in Section~\ref{sec:proofkeylemma}, we prove Lemma~\ref{lemma:arithmeticLemma}.

\subsection*{Acknowledgements}
We thank Juliette Bruce for the proof of Proposition \ref{prop:unboundedslicerank} which simplifies the proof we had in a previous version of this paper. We thank an anonymous referee for useful comments which helped us to improve the first version of our paper. A.O. thanks the University of Bern (Switzerland) for its hospitality during a visit where the project of the paper was discussed. During this project, A.B. was supported by the NWO Vici grant entitled \textit{Stabilisation in Algebra and Geometry} and A.O. acknowledges financial support from the Spanish Ministry of
Economy and Competitiveness, through the Mar\'ia de Maeztu Programme for Units of Excellence in R\&D (MDM-2014-0445) and from the Alexander von Humboldt-Stiftung (Germany) via a Humboldt Research Fellowship for Postdoctoral Researchers
(April~2019--March~2020).

\section{The slice rank of a form}\label{sec:sliceRanks}

As far as we know, the term \textit{slice rank} for tensors was explicitly introduced by Sawin and Tao in \cite{ST16:blog}. Here, we consider a symmetric version of this notion.

\begin{definition}\label{def:sliceRank}
Let $f \in \calS_d$. A \textbf{(symmetric) slice decomposition} of $f$ is an expression 
\[
f = \ell_1g_1 + \ldots + \ell_rg_r,
\]
where $\ell_i \in \calS_1, g_i \in \calS_{d-1}$. The minimal length $r$ of a slice decomposition of $f$ is the \textbf{(symmetric) slice rank} of $f$. We denote it by $\slrk(f)$.
\end{definition}

\begin{remark}\label{rmk:differencesSliceRk}
If $V_1,\ldots,V_d$ are $\Bbbk$-vector spaces, the slice rank of a tensor $t \in \bigotimes_j V_j$ is defined as follows. For each $i \in \{1,\ldots,d\}$, let $\otimes_i\colon V_i \times \bigotimes_{j \neq i} V_j \rightarrow \bigotimes_j V_j$ be the bilinear map sending $(v_i,v_1\otimes\cdots\otimes \hat{v}_i\otimes\cdots\otimes v_d)$ to $v_1\otimes\cdots\otimes v_d$. Here $\hat{v}_i$ denotes that $v_i$ is missing from the expression. Then, the slice rank of $t \in \bigotimes_j V_j$ is the smallest length $r$ of an expression 
\[
t = \sum_{k=1}^r v_k \otimes_{i_k} t_k,
\]
where $v_k \in V_{i_k}$ and $t_k\in \bigotimes_{j \neq i_k} V_j$ for all $k$. Since symmetric tensors are naturally identified with homogeneous polynomials, we have two notions of slice rank in this case, which are very different. Consider for example the symmetric tensor 
\[
t:=\frac{1}{d}(x_0\otimes x_1 \otimes \cdots \otimes x_1 + x_1 \otimes x_0 \otimes x_1 \otimes \cdots \otimes x_1 + \ldots + x_1 \otimes \cdots \otimes x_1 \otimes x_0)
\]
corresponding to the monomial $x_0x_1^{d-1}\in\calS_d$. The slice rank from Definition \ref{def:sliceRank} of the polynomial $x_0x_1^{d-1}$ is $1$, while the slice rank defined in \cite{ST16:blog} of the above tensor is equal to $2$. It is for this reason that we call the slice rank from Definition \ref{def:sliceRank} the \textit{symmetric} slice rank. However, in this paper we only consider homogeneous polynomials and therefore we can call it simply the slice rank since no ambiguity occurs.

We see that the tensor $t$ has slice rank $\leq2$ using the following decomposition
\[
	t = x_0 \otimes_1(x_1 \otimes \cdots \otimes x_1) + x_1 \otimes_1 \big( x_0 \otimes x_1 \otimes \cdots \otimes x_1 + x_1 \otimes x_0 \otimes \cdots \otimes x_1 + \ldots + x_1 \otimes \cdots \otimes x_1 \otimes x_0\big).
\]
To see that the tensor has slice rank $>1$, we note that any tensor $t \in \bigotimes_j V_j$ induces natural maps 
\[
\langle -,t\rangle_i\colon{\rm Hom}\left({\textstyle \bigotimes_{j \neq i}} V_j,\Bbbk\right)\to V_i,\quad i=1,\ldots,d
\]
with the property that the image of $\langle -,v \otimes_{i} t\rangle_i$ is spanned by $v$. For our tensor $t$, the images these maps are all spanned by $x_0$ and $x_1$. So the tensor must have slice rank $>1$.
\remarkend
\end{remark}

\begin{example}
By the \textit{Fundamental Theorem of Algebra}, any nonzero binary form ($n=1$) has slice rank equal to $1$. The slice rank is subadditive, i.e., $\slrk(\sum_i f_i) \leq \sum_i\slrk(f_i)$. It follows that the slice rank of a homogeneous polynomial in $n+1$ variables has slice rank $\leq n$. Indeed, any polynomial $f \in \calS_d$ can be written as 
\[
f = f(x_0,x_1,0,\ldots,0)+ x_2\cdot g_2 + \ldots + x_n\cdot g_n,\quad g_2,\ldots,g_n\in\calS_{d-1}.
\]
The same bound can be explained more geometrically: any point $P = \{\ell_1 = \ldots = \ell_n = 0\}$ on the hypersurface $\{f = 0\} \subseteq \bbP^n_\Bbbk$ provides a slice decomposition of $f$ with $n$ summands. This geometric interpretation of slice decompositions in terms of linear spaces contained in hypersurfaces is explained in Section \ref{sec:Fano_slicerank}.
\exampleend
\end{example}

\begin{example}
In the case of quadrics ($d=2$), if $\Bbbk$ is a field of characteristic different than $2$ and $i \in \Bbbk$ is an element such that $i^2 = -1$, then we have
\[
\ell_1\ell_2 = \left(\frac{1}{2} (\ell_1+\ell_2)\right)^2 + \left(\frac{i}{2} (\ell_1-\ell_2)\right)^2 \qquad \text{ and } \qquad \ell_1^2 + \ell_2^2 = (\ell_1 + i\ell_2)(\ell_1 - i\ell_2)
\]
for all linear forms $\ell_1,\ell_2$. Identifying quadrics with symmetric matrices of size $(n+1)\times(n+1)$, it follows that the space of polynomials of slice rank $\leq r$ coincides with the variety of symmetric matrices of rank~$\leq 2r$. Hence, the general slice rank in $\calS_2$ is $\left\lceil \frac{n+1}{2} \right\rceil$.\exampleend
\end{example}

\subsection{Fano varieties of linear spaces and the general slice rank}\label{sec:Fano_slicerank}

Given $f \in \calS_d$, let $X_f$ be the hypersurface $\{f = 0\}$ in the $n$-dimensional projective space $\bbP^n_\Bbbk$. From a slice decomposition of $f$ of length~$r$, it is immediate that $X_f$ contains a linear space of codimension $r$. Conversely, if the hypersurface~$X_f$ contains the linear space $\{\ell_1=\ldots=\ell_r=0\}$, then $f$ belongs to the ideal $(\ell_1,\ldots,\ell_r)$, i.e., it gives rise to a slice decomposition of $f$ of length $r$. Therefore we have 
\begin{equation}\label{eq:sliceRank=codimension}
	\slrk(f) = \min\left\{\mathrm{codim}(H) \,\middle|\, H \text{ linear space}, H \subseteq X_f\right\}.
\end{equation}
As a direct consequence of this interpretation of the slice rank, it is easy to prove that the set of homogeneous polynomials of bounded slice rank is an algebraic variety. Indeed, it is enough to consider the incidence variety $\Xi = \left\{(H,[f]) ~|~ f|_H = 0\right\} \subseteq \bbG(n-r,n) \times \bbP(\calS_d)$. Since the Grassmannian is a complete variety, the projection of $\Xi$ to the second factor is a closed map. For the same proof see for example \cite[Corollary 2]{ST16:blog} in the tensorial case or \cite[Proposition 2.2]{DES17:TopologicalCubic} in the case of cubic polynomials. 

It follows that:
\begin{enumerate}
\item The set of forms with slice rank $\leq r$ is Zariski-closed.
\item The slice rank of the general form in $\calS_d$ coincides with the maximal slice rank. We denote it by $\slrk_{n,d}^\circ$ and simply call it the \textit{general slice rank} in $\calS_d$.
\end{enumerate}
We point out these two facts because they also hold for the usual rank of matrices, but fail for its higher-order generalizations of tensor rank and symmetric rank. For this reason, there is no need to define the \textit{border slice rank} as in the case of other ranks. 

By \eqref{eq:sliceRank=codimension}, the notion of slice rank is related to the study of the \textit{Fano varieties} $F_k(X)$ parametrizing $k$-dimensional linear spaces contained in a hypersurface $X$. See \cite[Example 12.5]{Harris:First} or the recent survey \cite{CZ19:LinesConics}. Concretely, given $f \in \calS_d$, we have
\begin{equation}\label{eq:slicerank_Fano}
\slrk(f) = \min\left\{r \in \bbZ_{\geq0} \,\middle|\, F_{n-r}(X_f) \neq \emptyset \right\}.
\end{equation}

\begin{example}\label{example:Fermat}
Let $f = x_0^d + \ldots + x_n^d$ be the degree-$d$ Fermat polynomial. Then 
\[
\slrk(f)\leq \left\lceil \frac{n+1}{2} \right\rceil
\]
since $f$ can be written as the sum of the binary forms $x_{2i}^d+x_{2i+1}^d$ for $i = 0,\dots,\lfloor (n+1)/2\rfloor$ together with the binary form $x_n^d$ if $n$ is even. Conversely, since $X_f$ is smooth, the slice rank of $f$ is at least~$(n+1)/2$. Indeed, any smooth nondegenerate hypersurface in $\bbP^n_\Bbbk$ cannot contain a linear space of dimension $m$ such that $2m\geq n$; see \cite[Proposition 1]{Starr}. Therefore, $\slrk(f) = \left\lceil \frac{n+1}{2} \right\rceil$.\exampleend
\end{example}

\begin{remark}\label{re:differenceRCslicerank}
Note that, at least when also considering fields that are not algebraically closed, the slice rank of a form can go down when we extend the ground field $\Bbbk$. For an example of this, again consider a Fermat polynomial $f = x_0^d + \ldots + x_n^d$ of even degree $d\geq 2$. The only real point on $X_f$ is $0$, and so the slice rank of $f$ equals $n+1$ over $\bbR$, while the slice rank of $f$ over $\bbC$ equals $\left\lceil \frac{n+1}{2} \right\rceil$.\remarkend
\end{remark}

\begin{proposition}\label{prop:slicerankreducible}
Let $g,h$ be forms of degree $\geq 2$ and take $f=g\cdot h$. Then
\[
\slrk(f)=\min(\slrk(g),\slrk(h)).
\]
\end{proposition}
\begin{proof}
Note that $X_f=X_g\cup X_h$ and so a linear subspace, which is always irreducible, is contained in~$X_f$ if and only if it is contained in $X_g$ or $X_h$. The proposition now follows from (\ref{eq:sliceRank=codimension}).
\end{proof}

A numerical condition to guarantee the nonemptiness of the Fano scheme is well-known.

\begin{theorem}[{\cite[Theorem 12.8]{Harris:First}}]\label{thm:nonemptinessFano}
Let $n\geq1$, $d\geq 3$ and $r\geq0$ be positive integers and take
\[
\delta(n,d,r) := (r+1)(n-r) - \binom{d+r}{d}.
\]
Let $f \in \calS_d$ be a general form. 
\begin{enumerate}
\item If $\delta(n,d,r) \geq 0$, then $F_r(X_f)$ is nonempty, smooth and of dimension $\delta(n,d,r)$.
\item If $\delta(n,d,r) < 0$, then $F_r(X_f)$ is empty.
\end{enumerate}
\end{theorem}

Using Theorem \ref{thm:nonemptinessFano}, the general slice rank in $\calS_d$ can be computed.

\begin{corollary}\label{cor:maximalSlice}
Let $n \geq 1$ and $d \geq 3$ be integers. The general slice rank in $\calS_d$ is
\[
\slrk^{\circ}_{n,d} = \min\left\{r \in \bbZ_{\geq0} \,\middle|\, r(n+1-r) \geq \binom{d + n - r}{d}\right\}.
\]
\end{corollary}
\begin{proof}
Let $f \in \calS_d$ be a general form. By Theorem \ref{thm:nonemptinessFano}, $F_{n-r}(X_f) \neq \emptyset$ if and only if
\[
(n-r+1)r - \binom{d + n - r}{d} \geq 0.
\]
By \eqref{eq:slicerank_Fano}, this concludes the proof.
\end{proof}

\begin{example}
By \cite[Theorem 12.8]{Harris:First}, if $d > 2n-3$, then general degree-$d$ hypersurfaces in $\bbP^n_\Bbbk$ contain no lines. So in this case, it follows that the general slice rank is equal to $n$.\exampleend
\end{example}

\begin{remark}
If we consider a family of polynomials $\calF = \{f_1,\ldots,f_s\} \in \prod_{i=1}^s \calS_{d_i}$, then we can ask for the \textit{simultaneous} slice rank, that is, the minimal set of linear forms $\{\ell_1,\ldots,\ell_r\}$ such that there exist $g_{i,j}$ with $f_i = \sum_{j=1}^r \ell_jg_{i,j}$ for all $i$. This is the minimal codimension of a linear space contained in the variety $X_\calF := \{f_1 = \ldots = f_s = 0\}$. Generically, if $s \leq n+1$, the forms in $\calF$ may be assumed to form a regular sequence. This means that $X_\calF$ is a complete intersection. Fano varieties of linear spaces in complete intersections have also been studied extensively and we have again a result analogous to Theorem \ref{thm:nonemptinessFano}. See for example \cite[Corollary 2.2]{Bor90:DeformingVarieties} or the recent survey \cite[Theorem 2.6]{CZ19:LinesConics}. Consequently, the value of the general simultaneous slice rank is known: fix positive integers $d_1,\ldots,d_s \geq 3$ and $n \geq 1$, then the simultaneous slice rank of a general family $\calF \in \prod_{i=1}^s \calS_{d_i}$ is
\begin{equation}\label{eq:simultaneousSliceRank}
\slrk_{n,(d_1,\ldots,d_s)}^\circ := \min\left\{ r \in \bbZ_{\geq0} \,\middle|\, r(n+1-r) \geq \sum_{i=1}^s \binom{d_i +n -r}{ d_i}\right\}.
\end{equation}
In this case, the set of polynomial vectors with bounded simultaneous slice rank is also a variety. In particular, the right-hand-side of \eqref{eq:simultaneousSliceRank} is also the maximal simultaneous slice rank in $\prod_{i=1}^s \calS_{d_i}$.\remarkend
\end{remark}

\section{The strength of a form}\label{sec:strength_geometry}

Slice decompositions are special examples of \textit{strength decompositions}. 

\begin{definition}\label{def:strength}
Let $f \in \calS_d$. A \textbf{strength decomposition} of $f$ is an expression
\[
f = g_1h_1 + \ldots + g_rh_r,
\]
where $g_i \in \calS_{d_i}, h_i \in \calS_{d-d_i}$. The minimal length $r$ of a strength decomposition of $f$ is called the \textbf{strength} of $f$. We denote it by $\str(f)$.
\end{definition}

Clearly, we have $\str(f) \leq \slrk(f)$ for all $f\in\calS_d$. Moreover, when $d=2,3$ or $n=1$, the strength and slice rank coincide for every polynomial. Conjecture \ref{conj:genStr=genSl} says that generically this also holds for forms of higher degree. Note however that, in any degree $\geq 4$, the gap between strength and slice rank can be arbitrarily large for particular forms.

\begin{proposition}\label{prop:unboundedslicerank}
For integers $d \geq 4$ and $2 \leq i \leq d-2$, the set of slice ranks of forms $f\cdot g$ where $f,g$ are forms of degrees $i,d-i$, respectively, is unbounded. 
\end{proposition}
\begin{proof}
Consider the polynomial $g_{n,i} = f_{n,d}\cdot f_{n,d-i}$, where $f_{n,j}$ is the degree-$j$ Fermat polynomial in $n+1$ variables. As we have seen in Example \ref{example:Fermat}, the polynomial $f_{n,j}$ has slice rank equal to  $\left\lceil \frac{n+1}{2}\right\rceil$. Therefore, while the strength of $g_{n,i}$ is equal to $1$, we have that the slice rank is~$\left\lceil \frac{n+1}{2}\right\rceil$ by Proposition~\ref{prop:slicerankreducible}.
\end{proof}

\begin{example}\label{example:Fermat2}
Again, let $f = x_0^d + \ldots + x_n^d$ be the degree-$d$ Fermat polynomial. Then 
\[
\str(f)\leq\slrk(f)=\left\lceil \frac{n+1}{2} \right\rceil.
\]
As pointed out in the introduction of \cite{AH:Stillman}, the partial derivatives of a polynomial $f = \sum_{i=1}^r g_ih_i$ are contained in the ideal $(g_1,h_1,\ldots,g_r,h_r)$ with $2r$ generators. So we get the same lower bound for the strength of forms defining smooth nondegenerate hypersurfaces as for the slice rank. Therefore also $\str(f) = \left\lceil \frac{n+1}{2} \right\rceil$.\exampleend
\end{example}

We do not know whether, like the slice rank, the strength of even degree Fermat polynomials over $\bbR$ and over $\bbC$ differ (in even degree $d\geq 4$). 

The notion of strength behaves very differently than the one of slice rank. For example, if $f = \sum_{i=1}^r g_ih_i$, then the variety $\{g_1 = \ldots =g_r = 0\}$ is clearly contained in $X_f$, but there is no reason to expect that it has codimension $r$ since the $g_i$ might not define a complete intersection if they have degrees higher than $2$. In the paper \cite{BBOV:notClosed}, together with Ballico and Ventura, we underline another big difference as we show that the set of forms with bounded strength is not necessarily Zariski-closed.

\subsection{Secant varieties of varieties of reducible forms}\label{sec:geom_strength}
In \cite{CCG08:CIinHypersurfaces}, the authors considered strength decompositions in order to understand which complete intersections can be contained in general hypersurfaces. This is motivated by the following remark obtained from a private communication with E.~Ballico and E. Ventura.

\begin{remark}
Although a decomposition $f = \sum_{i=1}^r g_ih_i$ does not always imply that the subvariety 
\[
\{g_1 = \ldots = g_r = 0\} \subseteq X_f
\]
is a complete intersection, it may be assumed if we consider a general form: let $r$ be the general strength and $U \subseteq \bbP(\calS_d)$ be the dense subset of forms of strength $r$. Let $U' \subseteq U$ be the subset of forms having a strength decomposition defined by a complete intersection. Then, $U'$ is dense in $U$ and, in particular, the general hypersurface contains a complete intersection of codimension $r$. Indeed, let $f = \sum_{i=1}^r g_ih_i \in U \setminus U'$ with $\deg(g_i) = d_i$. Then, since the set of regular sequences is dense, there is {a} $(u_1,\ldots,u_r) \in \calS_{d_1}\times \cdots \times \calS_{d_r}$ and an $\epsilon>0$ such that $(g_1+tu_1,\ldots,g_r+tu_r)$ {is a regular sequence for each} $t \in (0,\epsilon]$. In particular, we have $f = \lim_{t \rightarrow 0} f_t$  where 
\[
f_t = \sum_{i=1}^r \left(g_i+tu_i\right)h_i \in U'.
\]
\remarkend
\end{remark}

In other words, where the general slice rank measures the smallest codimension of a linear space contained in a general hypersurface, the general strength measures the smallest codimension of a complete intersection contained in a general hypersurface.
	
In \cite{CCG08:CIinHypersurfaces}, the authors approach the problem by studying secant varieties of the varieties of reducible forms. Here, we do the same. For $i = 1,\ldots,\left\lfloor d/2 \right\rfloor$, consider the \textit{variety of degree-$d$ forms with a factor of degree~$i$}, i.e., the variety
$X_{(i,d-i)} := \{ [g\cdot h] \mid[g] \in \bbP(\calS_i), [h] \in\bbP(\calS_{d-i})\} \subseteq \bbP(\calS_d)$. We define the \textbf{variety of reducible forms} as their union:
\[
X_{\rm red} := \bigcup_{i=1}^{\lfloor d/2\rfloor} X_{(i,d-i)} \subseteq \bbP(\calS_d).
\]
Note that 
\[
\dim X_{(i,d-i)} = \dim\bbP(\calS_i)+\dim\bbP(\calS_{d-1})=\binom{n+i}{n} + \binom{n+d-i}{n} -2 
\]
and $\dim X_{\rm red} = \dim X_{(1,d-1)}$; see \cite[Proposition~7.2]{C+19:SecantReducible}.

In order to give a better geometrical description of the variety of forms of bounded (border) strength, we recall the definitions of the \textit{join} of algebraic varieties and of \textit{secant varieties}. Given algebraic varieties $X_1,\ldots,X_r \subseteq \bbP^n_\Bbbk$, the \textit{join} $J(X_1,\ldots,X_r)$ of $X_1,\ldots,X_r$ is the Zariski-closure of the union of all linear spaces spanned by $r$-tuples of distinct points in $X_1\times\cdots\times X_r$. In the particular case where $X_1 = \ldots = X_r = X$, the join $\sigma_r(X):= J(X,\ldots,X)$ is called the \textit{$r$-th secant variety} of $X$. 

By definition, we see that $\sigma_r(X_{\rm red}) = \overline{\left\{[f] \in \bbP(\calS_d) \,\middle|\, \str(f) \leq r \right\}} \subseteq \bbP(\calS_d)$. The variety of reducible forms is highly reducible: we have 
\begin{equation}\label{eq:reducibleFormsJoins}
\sigma_r(X_{\rm red}) = \bigcup_{1\leq a_1\leq\cdots\leq a_r\leq\lfloor d/2\rfloor} J(X_{(a_1,d-a_1)},\ldots,X_{(a_r,d-a_r)}).
\end{equation}
The general strength $r$ corresponds to the first secant variety $\sigma_r(X_{\rm red})$ that fills the ambient space. 

\begin{remark}\label{rmk:reducibility}
From the proof of Proposition \ref{prop:unboundedslicerank}, we deduce that for {$1\leq r< \left\lceil \frac{n+1}{2} \right\rceil$ the variety $\sigma_r(X_{\rm red})$ is not contained in $\sigma_r(X_{(1,d-1)})$, and hence reducible. Indeed, consider $g = f_{{n},i} \cdot f_{{n},d-i}$. The slice rank of $g$ is equal to $\left\lceil \frac{n+1}{2} \right\rceil>r$ {and, since the set of forms of bounded rank is Zariski-closed, we deduce $g \not\in\sigma_r(X_{(1,d-1)})$. However, the strength of $g$ is clearly equal to $1\leq r$; hence, we have $g \in X_{(i,d-i)} \subseteq \sigma_r(X_{(i,d-i)})$.} This idea can be extended by considering products $g\cdot h$ where $g,h$ are forms with generic slice ranks in degrees $i,d-i$, respectively. Note however that this idea cannot prove the reducibility of $\sigma_r(X_{\rm red})$ for $\max\{\slrk_{2,n}^\circ,\ldots,\slrk_{\lfloor d/2\rfloor,n}^\circ\}\leq r<\slrk_{d,n}^\circ$, which is a nonempty range in general.}
\remarkend
\end{remark}

In \cite[Remark 7.7]{C+19:SecantReducible}, the authors conjecture the following.

\begin{conjecture*}[Conjecture \ref{conj:dimEquality_intro}]
For every $r$, the dimension of the $r$-th secant variety of the variety of reducible forms is equal to the dimension of the component given by the $r$-th secant variety of the variety of forms with a linear factor. I.e.,
\[
\dim \sigma_r (X_{\rm red}) = \dim \sigma_r (X_{(1,d-1)}).
\]
\end{conjecture*} 
In \cite[Theorem 7.4]{C+19:SecantReducible}, the authors prove this conjecture for $2r \leq n$. Conjecture \ref{conj:dimEquality_intro} implies Conjecture \ref{conj:genStr=genSl} on the general equality of slice rank and strength. Indeed, the general slice rank is equal to $r$ if and only if the $r$-th secant variety of $X_{\rm red}$ is the first one filling the ambient space and, assuming Conjecture \ref{conj:dimEquality_intro}, this is equivalent to saying that the $r$-th secant variety of $X_{(1,d-1)}$ is the first one filling the ambient space, i.e., the general form has slice rank equal to $r$.

Following a standard approach to studying dimensions of secant varieties, we look at the tangent spaces of the components to $\sigma_r(X_{\rm red})$ at general points. We can compute these spaces using the classical \textit{Terracini's Lemma}.

\begin{lemma}[Terracini's Lemma, \cite{Ter11}]
Let $X_1,\ldots,X_r$ be algebraic varieties. Let $p_1\in X_1$, \ldots , $p_r\in X_r$ and $q \in \langle p_1,\ldots,p_r\rangle$ be general points. Then $T_qJ(X_1,\ldots,X_r) = \langle T_{p_1}(X_1),\ldots,T_{p_r}(X_r)\rangle$.
\end{lemma}	

Consider the parametrization of $X_{(a,d-a)}$
\begin{eqnarray*}
\varphi\colon \calS_a \times \calS_{d-a} &\to& \calS_d\\
(g,h) &\mapsto& g\cdot h
\end{eqnarray*}
and fix a point $p = (g,h) \in \calS_a \times \calS_{d-a}$. For a tangent direction $(g',h')$, consider the line 
\[
L_p(t) = (g+tg',h+th').
\]
The tangent direction to the curve $\varphi(L_p(t))$ through the point $\varphi(p) = gh$ is $gh' + g'h$, indeed
\[
\left.\left(\frac{d}{dt} (g+tg')(h+th')\right)\right|_{t=0} = g'h + gh'.
\]
Therefore, the tangent space at $[g\cdot h] \in X_{(a,d-a)}$ is $T_{[gh]}X_{(a,d-a)} = \bbP((g,h)_{d})$ where $(g,h)_{d}$ is the homogeneous degree-$d$ part of the ideal $(g,h)$. Hence, by Terracini's Lemma, if $q$ is a general point on $J(X_{(a_1,d-a_1)},\ldots,X_{(a_r,d-a_r)})$, then
\begin{equation}\label{eq:tangentJoin}
T_qJ(X_{(a_1,d-a_1)},\ldots,X_{(a_r,d-a_r)}) = \bbP((g_1,h_1,\ldots,g_r,h_r)_{d}),
\end{equation}
where the $g_i$ and $h_i$ are general. So, in order to prove Conjecture~\ref{conj:dimEquality_intro}, we study the dimensions of homogeneous components of ideals generated by general forms or, equivalently, their \textit{Hilbert functions}.

\begin{remark}\label{rmk:HL}
The $r$-th secant variety $\sigma_r(X_{(1,d-1)})$ of the variety of forms with a linear factor is the variety of forms with slice rank $\leq r$. From Terracini's Lemma, we know that the tangent space to $\sigma_r(X_{(1,d-1)})$ at a general point corresponds to $(\ell_1,g_1,\ldots,\ell_r,g_r)_{d}$ where the $\ell_i$ are linear and the $g_i$ have degree~$d-1$. A result by Hochster and Laksov \cite{HL87:LinearSyzygies} states that ideals generated by general forms of degree $d-1$ do not have linear syzygies. As a consequence of this, it is possible to compute the dimension of $(\ell_1,g_1,\ldots,\ell_r,g_r)_{d}$ and hence of all secant varieties $\sigma_r(X_{(1,d-1)})$; see \cite[Proposition~5.6]{CCG08:CIinHypersurfaces}.\remarkend
\end{remark}

\section{Hilbert functions of general ideals and Fr\"oberg's conjecture}\label{sec:Froberg}

Given a homogeneous ideal $I \subseteq \calS$, the \textit{Hilbert function} of $\calS/I$ is the numerical function
\begin{eqnarray*}
\HF_{\calS/I}\colon \bbZ_{\geq0} &\to& \bbZ_{\geq0}\\
d &\mapsto& \dim (\calS/I)_{d} =:\HF_{\calS/I}(d)
\end{eqnarray*}
and the \textit{Hilbert series} of $\calS/I$ is the generating power series $\HS_{\calS/I}(t) := \sum_{d\geq0} \HF_{\calS/I}(d)t^d \in \bbZ[[t]]$.	Hilbert series are among the most interesting and well-studied algebraic invariants associated to an homogeneous ideal since they encode a lot of information of the algebraic variety defined by it. By~\eqref{eq:tangentJoin},
\[
\codim J(X_{(a_1,d-a_1)},\ldots,X_{(a_r,d-a_r)}) = \HF_{\calS/I}(d),
\]
where $I = (g_1,h_1,\ldots,g_r,h_r)$ is generated by general forms with $\deg(g_i) = a_i$ and $\deg(h_i) = d-a_i$.

The Hilbert series for ideals generated by general forms is prescribed by \textit{Fr\"oberg's conjecture}. 

\begin{notation}
Let $P=\sum_{i\geq 0} a_it^i\in\bbZ[[t]]$ and take
\[
b_i = 
\begin{cases}
a_i & \text{ if } a_j \geq 0, \text{ for } j \leq i \\
0 & \text{ otherwise}
\end{cases}
\] 
for every $i\geq0$. Then we write  $\coeff_d(P) := a_d$ for each integer $d\geq0$ and $\lceil P \rceil := \sum_{i\geq 0} b_i t^i$.
\end{notation}

\begin{conjecture}[{Strong Fr\"oberg's Conjecture (sFC), \cite{Fr85:HilbertSeriesGradedAlg}}]\label{conj:Froberg}
Let $f_1,\ldots,f_s$ be general forms in $n+1$ variables of degrees $d_1,\ldots,d_s$ and take $I = (f_1,\ldots,f_s)\subseteq\calS$. Then for each integer $d\geq0$, we have
\begin{equation}\label{eq:sFC}
\coeff_d\left(\HS_{\calS/I}(t)\right) = \coeff_d \left(\left\lceil \frac{\prod_{i=1}^s (1-t^{d_i})}{(1-t)^{n+1}}\right\rceil\right).
\end{equation}
\end{conjecture}

When the number of generators $s$ is at most the number of variables $n+1$, a general ideal is a complete intersection and it is an easy exercise in commutative algebra to see that the formula holds with no need for brackets. The solution of the case $s = n+2$ is attributed to Stanley; see \cite[Proposition at pg.~367]{Ia:Compressed}. The sFC is also known in a few more cases: in two variables \cite{Fr85:HilbertSeriesGradedAlg}; in three variables \cite{An:Froberg}; and in degree $\min_i\{d_i\}+1$ \cite{HL87:LinearSyzygies}. Evidence pointing towards the conjecture is also given by an asymptotic result in \cite{Ne17}. See \cite{BFL18:surveyHF} for a recent survey on these questions.

In \cite{Ia97:InverseSystem3_ThinAlgebras}, a weaker version of Fr\"oberg's conjecture is also considered. 

\begin{conjecture}[{Weak Fr\"oberg's Conjecture (wFC), \cite{Ia97:InverseSystem3_ThinAlgebras}}]\label{conj:wFroberg}
Let $f_1,\ldots,f_s$ be general forms in $n+1$ variables of degrees $d_1,\ldots,d_s$ and take $I = (f_1,\ldots,f_s)\subseteq\calS$. Then for each integer $d\geq0$, we have
\begin{equation}\label{ineq:wFC}
\coeff_d\left(\HS_{\calS/I}(t)\right) \geq \coeff_d \left(\left\lceil \frac{\prod_{i=1}^s (1-t^{d_i})}{(1-t)^{n+1}}\right\rceil\right).
\end{equation}
\end{conjecture}

Not much more is known regarding wFC compared to sFC. In \cite{Fr85:HilbertSeriesGradedAlg}, Fr\"oberg proved that the inequality holds \textit{lexicographically}, i.e., it holds at the first coefficient where equality fails. It is trivial to notice that wFC holds for $d<2\min_i\{d_i\}$, since we have no Koszul syzygies in this range. Following an idea of Iarrobino \cite{Ia97:InverseSystem3_ThinAlgebras}, our approach essentially is to prove instances of wFC from instances of sFC and wFC for fewer generators.

\section{Proof of the main results}\label{sec:proofMain}

Fix integers $n\geq1$, $d\geq 2$ and $\ell_1,\ell_2,\ldots,\ell_{\left\lfloor d/2 \right\rfloor}\geq0$. Write $r=\ell_1+\ell_2+\ldots+\ell_{\left\lfloor d/2 \right\rfloor}$ and
\[
J_{\ell_1,\ldots,\ell_{\left\lfloor d/2 \right\rfloor}} := J\left(\sigma_{\ell_1}(X_{(1,d-1)}),\sigma_{\ell_2}(X_{(2,d-2)}),\ldots,\sigma_{\ell_{\lfloor d/2\rfloor}}(X_{({\lfloor d/2\rfloor},{\lceil d/2\rceil})})\right).
\]
Then, we want to show that 
\begin{equation}\label{eq:main}
\dim J_{\ell_1,\ldots,\ell_{\left\lfloor d/2 \right\rfloor}}  \leq \dim \sigma_r(X_{(1,d-1)}).
\end{equation}
Write $m(\ell) := n -\ell$ and take
\[
f_{n,d}(m,\ell_2,\ldots,\ell_{\left\lfloor d/2 \right\rfloor}) := \sum_{\beta\in\bbZ_{\geq0}^{2\times(\lfloor d/2\rfloor-1)}}(-1)^{|\beta|}\prod_{i=2}^{\lfloor d/2\rfloor}\binom{\ell_i}{\beta_{1,i}}\binom{\ell_i}{\beta_{2,i}}\cdot\binom{m+d-||\beta||}{m} - (n-m)(m+1),
\]
where we write $||\beta||:=\sum_{i=2}^{\lfloor d/2\rfloor}\left(\beta_{1,i}\cdot i+\beta_{2,i}\cdot(d-i)\right)$ and $|\beta|:=\sum_{i=2}^{\lfloor d/2\rfloor}\left(\beta_{1,i}+\beta_{2,i}\right)$ for all
\[
\beta=\begin{pmatrix}\beta_{1,2}&\beta_{1,3}&\cdots&\beta_{1,\lfloor d/2\rfloor}\\\beta_{2,2}&\beta_{2,3}&\cdots&\beta_{2,\lfloor d/2\rfloor}\end{pmatrix}\in\bbZ_{\geq0}^{2\times(\lfloor d/2\rfloor-1)}.
\]
It is direct to show that
\begin{align*}
\coeff_d \left(\frac{\prod_{i=1}^{\left\lfloor d/2 \right\rfloor} (1-t^i)^{\ell_i}(1-t^{d-i})^{\ell_i}}{(1-t)^{n+1}}\right) & = \coeff_d \left( \frac{(1-t^{d-1})^{\ell_1}\prod_{i=2}^{\left\lfloor d/2 \right\rfloor} (1-t^i)^{\ell_i}(1-t^{d-i})^{\ell_i}}{(1-t)^{m(\ell_1)+1}}\right) \\
& = \coeff_d \left( \frac{\prod_{i=2}^{\left\lfloor d/2 \right\rfloor} (1-t^i)^{\ell_i}(1-t^{d-i})^{\ell_i}}{(1-t)^{m(\ell_1)+1}}\right) -\ell_1(m(\ell_1)+1)\\
& = f_{n,d}(m(\ell_1),\ell_2,\ldots,\ell_{\left\lfloor d/2 \right\rfloor})
\end{align*}
As explained in Remark \ref{rmk:HL}, we have 
\[
\codim\sigma_r(X_{(1,d-1)}) = \coeff_d\left(\frac{(1-t^{d-1})^r}{(1-t)^{n-r}}\right).
\]
So in particular, we get
\begin{equation}\label{eq:codim_1}
\codim\sigma_r(X_{(1,d-1)}) = f_{n,d}(m(r),0,\ldots,0).
\end{equation}
Now, to prove Theorem~\ref{thm:main}, we use the following lemma.

\begin{lemma}\label{lemma:arithmeticLemma}
Let $d \leq 10$. In the same notation as above, if $r=\ell_1 + \ldots + \ell_{\lfloor d/2\rfloor} \leq \slrk_{n,d}^\circ - 1$, then
\[
f_{n,d}(m(\ell_1),\ell_2,\ldots,\ell_{\lfloor d/2\rfloor}) \geq f_{n,d}(m(r),0,\ldots,0).
\]
{If $\ell_2+\ldots+\ell_{\lfloor d/2\rfloor}>0$ and $(n,d,r)\neq (3,4,2)$, then the inequality holds strictly.}
\end{lemma}

Before getting into the technicalities of the proof of this lemma in the next section, we show that Theorem~\ref{thm:main} follows from Lemma \ref{lemma:arithmeticLemma}. The key will be to prove that
\begin{equation}\label{ineq:dimJoins}
\codim J_{\ell_1,\ldots,\ell_{\left\lfloor d/2 \right\rfloor}} \geq f_{n,d}(m(\ell_1),\ell_2,\ldots,\ell_{\lfloor d/2\rfloor}).
\end{equation}
This inequality, combined with Lemma \ref{lemma:arithmeticLemma} and \eqref{eq:codim_1}, shows that
\begin{equation}\label{ineq:codimensionsJoins}
\codim J_{\ell_1,\ldots,\ell_{\left\lfloor d/2 \right\rfloor}} \geq \codim \sigma_r(X_{(1,d-1)})
\end{equation}
whenever $r=\ell_1 + \ldots + \ell_{\lfloor d/2\rfloor}\leq \slrk_{n,d}^\circ - 1$, i.e., the assertion of Conjecture \ref{conj:dimEquality_intro}. Note here that for $r\geq  \slrk_{n,d}^\circ$ the variety $\sigma_r(X_{(1,d-1)})$ equals the entire space and so Conjecture \ref{conj:dimEquality_intro} holds trivially.

\begin{remark}
In \cite[Lemma 7.3]{C+19:SecantReducible}, the authors prove Conjecture \ref{conj:dimEquality_intro} when $2r \leq n+1$. Under this assumption, \eqref{ineq:dimJoins} holds with equality (see Lemma \ref{lemma:P4lows}), while for larger numbers of generators, not even the inequality is known to be true in general. Under this numerical assumption, Lemma \ref{lemma:arithmeticLemma} also gets easier: indeed it can be deduced from the case $2r = n+1$ where the inequality is deduced from the immediate inequality
	\[
		\coeff_j \frac{(1-t^i)(1-t^{d-i})}{(1-t)^2} \geq \coeff_j \frac{(1-t)(1-t^{d-1})}{(1-t)^2}, \quad \text{ for } 1 \leq j \leq d/2. 
	\] 
Hence, we focus on proving \eqref{ineq:dimJoins}.\remarkend
\end{remark}

Let $f_1,\ldots,f_s$ be general forms in $n+1$ variables of degrees $d_1,\ldots,d_s$ and take $I = (f_1,\ldots,f_s)$. Then we write:
\begin{itemize}
\item $\wFC_{n,d}(d_1,\ldots,d_s)$ (resp. $\sFC_{n,d}(d_1,\ldots,d_s)$) if the inequality \eqref{ineq:wFC} (resp. equality~\eqref{eq:sFC}) holds. 
\item $\mathrm{P}_{n,d}(d_1,\ldots,d_s)$ to indicate that the inequality 
\begin{equation}\label{eq:propertyP}
\coeff_d\left(\HS_{\calS/I}(t)\right) \geq \coeff_d \left(\frac{\prod_{i=1}^s (1-t^{d_i})}{(1-t)^{n+1}}\right)
\end{equation}
holds.
\item $\mathrm{Q}_{n,d}(d_1,\ldots,d_s)$ to indicate that the inequality
\begin{equation}\label{eq:propertyQ}
\coeff_d\left(\HS_{\calS/I}(t)\right) \leq \coeff_d \left(\frac{\prod_{i=1}^s (1-t^{d_i})}{(1-t)^{n+1}}\right)
\end{equation}
holds.
\end{itemize}
Using this notation, we want to prove instances of the property P.

\begin{remark}
Note that $\mathrm{P}_{n,d}(d_1,\ldots,d_s)$ implies $\wFC_{n,d}(d_1,\ldots,d_s)$. Also, if 
\[
\coeff_d \left(\left\lceil \frac{\prod_{i=1}^s (1-t^{d_i})}{(1-t)^{n+1}}\right\rceil\right)=\coeff_d \left( \frac{\prod_{i=1}^s (1-t^{d_i})}{(1-t)^{n+1}}\right),
\]
then $\sFC_{n,d}(d_1,\ldots,d_s)$ implies $\mathrm{P}_{n,d}(d_1,\ldots,d_s)$ and $\mathrm{Q}_{n,d}(d_1,\ldots,d_s)$.\remarkend
\end{remark}

\begin{lemma}\label{lemma:wFC}
If $\mathrm{Q}_{n,d-d_s}(d_1,\ldots,d_{s-1})$ and $\mathrm{P}_{n,d}(d_1,\ldots,d_{s-1})$ hold, then $\mathrm{P}_{n,d}(d_1,\ldots,d_s)$ holds as well.
\end{lemma}
\begin{proof}
Take $I' = (f_1,\ldots,f_{s-1})$. Then we have the short exact sequence
\[
(\calS/I')_{d-d_s} \xrightarrow{-\cdot g_s} (\calS/I')_d \rightarrow (\calS/I)_d \rightarrow 0.
\]
It follows that
\begin{align*}
\coeff_d\left(\HS_{\calS/I}(t)\right) & \geq \coeff_d\left(\HS_{\calS/I'}(t)\right) -  \coeff_{d-d_s}\left(\HS_{\calS/I'}(t)\right)  \\
& \geq \coeff_d \left( \frac{\prod_{i=1}^{s-1} (1-t^{d_i})}{(1-t)^{n+1}}\right) -  \coeff_{d-d_s} \left( \frac{\prod_{i=1}^{s-1} (1-t^{d_i})}{(1-t)^{n+1}}\right) \\
& = \coeff_d \left( \frac{\prod_{i=1}^{s} (1-t^{d_i})}{(1-t)^{n+1}}\right).
\end{align*}
\end{proof}

\begin{samepage}
\begin{lemma}\label{lemma:P4lows}
If $s\leq n+1$, then $\mathrm{Q}_{n,d}(d_1,\ldots,d_s)$ and $\mathrm{P}_{n,d}(d_1,\ldots,d_s)$ hold.
\end{lemma}
\begin{proof}
As sFC holds for $s\leq n+1$ without need for the brackets around the power series, both the properties $Q$ and $P$ hold.
\end{proof}
\end{samepage}

In our case, many of the $d_i$ are equal. So we use the following notation.

\begin{notation}
We write 
\[
(a_1^{(m_1)},\ldots,a_s^{(m_s)}) := (\underbrace{a_1,\ldots,a_1}_{m_1},\ldots,\underbrace{a_s,\ldots,a_s}_{m_s})
\]
for all $a_1,\ldots,a_s\in\bbN$ and $m_1,\ldots,m_s\in\bbZ_{\geq0}$.
\end{notation}
Let $d\geq 1$ and $m_1,\ldots,m_d\in\bbZ_{\geq0}$ be integers and take $I=I_1+\ldots+I_d$ where $I_i\subseteq\calS$ is an ideal generated by $m_i$ general forms of degree $i$. Our goal is to prove that $\mathrm{P}_{n,d}(1^{(m_1)},\ldots,d^{(m_d)})$ holds in the cases we need. Note that we can always assume to have no linear generators since, by genericity, they can always be removed in exchange for a decrease in the number of variables. More precisely, we have $\calS/I\cong (\calS/I_1)/(I_2+\ldots+I_s)$ and 
\[
\frac{\prod_{i=1}^d (1-t^i)^{m_i}}{(1-t)^{n+1}}=\frac{\prod_{i=2}^d (1-t^i)^{m_i}}{(1-t)^{n-m_1+1}}.
\]
It follows that $\mathrm{P}_{n,d}(1^{(m_1)},\ldots,d^{(m_d)})$ and $\mathrm{P}_{n-m_1,d}(2^{(m_2)},\ldots,d^{(m_d)})$ are equivalent. For the property $\mathrm{Q}$, we have the same equivalence.

\begin{lemma}\label{lemma:sFC2Q2}
Suppose that $m_2 \leq \binom{n+2}{2}$. Then $\mathrm{Q}_{n,2}(2^{(m_2)},\ldots,d^{(m_d)})$ holds.
\end{lemma}
\begin{proof}
We have
\[
\coeff_2\left(\HS_{\calS/I}(t)\right)=\max\left(\binom{n+2}{2}-m_2,0\right)=\binom{n+2}{2}-m_2=\coeff_2 \left( \frac{\prod_{i=2}^d (1-t^i)^{m_i}}{(1-t)^{n+1}}\right)
\]
since $m_2\leq\binom{n+2}{2}$.
\end{proof}

\begin{lemma}\label{lemma:sFC2Q3}
Suppose that $m_2\leq\binom{n+2}{2}$ and $m_2(n+1) + m_3\leq \binom{n+3}{3}$. Then $\mathrm{Q}_{n,3}(2^{(m_2)},\ldots,d^{(m_d)})$ holds.
\end{lemma}
\begin{proof}
We have
\[
\coeff_1 \left( \frac{\prod_{i=2}^d (1-t^i)^{m_i}}{(1-t)^{n+1}}\right)=n+1\geq{ 0},  \quad
\coeff_2 \left( \frac{\prod_{i=2}^d (1-t^i)^{m_i}}{(1-t)^{n+1}}\right)=\binom{n+2}{2}-m_2\geq{0}
\]
and
\[
\coeff_3 \left( \frac{\prod_{2=1}^d (1-t^i)^{m_i}}{(1-t)^{n+1}}\right)=\binom{n+3}{3}-m_2(n+1)-m_3\geq0.
\]
Hence 
\[
\coeff_3 \left(\left\lceil \frac{\prod_{i=2}^d (1-t^i)^{m_i}}{(1-t)^{n+1}}\right\rceil\right)=\coeff_3 \left( \frac{\prod_{i=2}^d (1-t^i)^{m_i}}{(1-t)^{n+1}}\right).
\]
Since it is a consequence of \cite[Theorem 1]{HL87:LinearSyzygies} that $\sFC_{n,3}(2^{(m_2)},\ldots,d^{(m_d)})$ holds, the lemma follows.
\end{proof}

\begin{proposition}\label{prop:deg4}
Suppose that $m_2-1\leq\binom{n+2}{2}$. Then $\mathrm{P}_{n,4}(2^{(m_2)},3^{(m_3)},4^{(m_4)})$ holds.
\end{proposition}
\begin{proof}
Write $I = (f_1,\ldots,f_{m_2},g_1,\ldots,g_{m_3},h_1,\ldots,h_{m_4})$, where $\deg(f_i) = 2, \deg(g_i) = 3, \deg(h_i) = 4$. Then $I= I' + J$ where $I' = (f_1,\ldots,f_{m_2})$ and $J=(g_1,\ldots,g_{m_3},h_1,\ldots,h_{m_4})$. By \cite[Theorem 1]{HL87:LinearSyzygies}, $\dim J_4 = m_3(n+1)+m_4$. Therefore, we deduce that
\[
\dim I_4 \leq \dim I'_4 + \dim J_4 = \dim I'_4 + m_3(n+1)+m_4,
\]
i.e.,
\[
\coeff_4\left(\HS_{\calS/I}(t)\right)\geq \coeff_4\left(\HS_{\calS/I_2}(t)\right)-m_3(n+1)-m_4.
\]
So since
\[
\coeff_4 \left(\frac{\prod_{i=2}^4 (1-t^i)^{m_i}}{(1-t)^{n+1}}\right)= \coeff_4 \left(\frac{(1-t^2)^{m_2}}{(1-t)^{n+1}}\right)-m_3(n+1)-m_4,
\]
it suffices to prove that $\mathrm{P}_{n,4}(2^{(m_2)})$ holds. We do this using induction on $m_2$. Note that $\mathrm{P}_{n,4}(2^{(0)})$ holds and that $\mathrm{P}_{n,4}(2^{(m_2)})$ follows from $\mathrm{P}_{n,4}(2^{(m_2-1)})$ and $\mathrm{Q}_{n,2}(2^{(m_2-1)})$ for $m_2>0$ by Lemma~\ref{lemma:wFC}. Since $m_2-1\leq\binom{n+2}{2}$, by Lemma~\ref{lemma:sFC2Q2}, the second condition is satisfied in every step. So the proposition follows by induction.
\end{proof}

\begin{proposition}\label{prop:deg5}
Suppose that $(m_2-1)(n+1)\leq\binom{n+3}{3}$. Then $\mathrm{P}_{n,5}(2^{(m_2)},3^{(m_3)},4^{(m_4)},5^{(m_5)})$ holds. 
\end{proposition}
\begin{proof}
With the same argument as in the previous proposition,  it is enough to prove the case $m_4 = m_5 = 0$. We prove that $\mathrm{P}_{n,5}(2^{(m_2)},3^{(m_3)})$ holds using induction on $m_2$ and $m_3$. Note that $\mathrm{P}_{n,5}(2^{(0)})$ holds and that $\mathrm{P}_{n,5}(2^{(m_2)})$ follows from $\mathrm{P}_{n,5}(2^{(m_2-1)})$ and $\mathrm{Q}_{n,3}(2^{(m_2-1)})$ for $m_2>0$ by Lemma~\ref{lemma:wFC}. By Lemma~\ref{lemma:sFC2Q3}, the second condition  is satisfied in every step. So $\mathrm{P}_{n,5}(2^{(m_2)})$ holds. Similarly, note that $\mathrm{P}_{n,5}(2^{(m_2)},3^{(m_3)})$ follows from $\mathrm{P}_{n,5}(2^{(m_2)},3^{(m_3-1)})$ and $\mathrm{Q}_{n,2}(2^{(m_2)},3^{(m_3-1)})$ for $m_3>0$ by Lemma~\ref{lemma:wFC}. By Lemma~\ref{lemma:sFC2Q2}, the second condition  is satisfied in every step. So the proposition follows by induction.
\end{proof}

\begin{proposition}\label{prop:deg6}
Suppose that $m_2\leq n+1$ and $m_2(n+1)+m_3-1 \leq \binom{n+3}{3}$. Then 
\[
\mathrm{P}_{n,6}(2^{(m_2)},\ldots,6^{(m_6)})
\]
holds.
\end{proposition}
\begin{proof}
With the same argument as in Proposition \ref{prop:deg4}, it is enough to prove the case $m_5 = m_6= 0$. We prove that $\mathrm{P}_{n,6}(2^{(m_2)},3^{(m_3)},4^{(m_4)})$ holds using induction on $m_3$ and $m_4$. Note that $\mathrm{P}_{n,6}(2^{(m_2)})$ holds by Lemma~\ref{lemma:P4lows}, that $\mathrm{P}_{n,6}(2^{(m_2)},3^{(m_3)})$ follows from $\mathrm{P}_{n,6}(2^{(m_2)},3^{(m_3-1)})$ and $\mathrm{Q}_{n,3}(2^{(m_2)},3^{(m_3-1)})$ for $m_3>0$ by Lemma~\ref{lemma:wFC} and that $\mathrm{P}_{n,6}(2^{(m_2)},3^{(m_3)},4^{(m_4)})$ follows from $\mathrm{P}_{n,6}(2^{(m_2)},3^{(m_3)},4^{(m_4-1)})$ and $\mathrm{Q}_{n,2}(2^{(m_2)},3^{(m_3)},4^{(m_4-1)})$ for $m_4>0$ by Lemma~\ref{lemma:wFC}. By Lemmas~\ref{lemma:sFC2Q2} and~\ref{lemma:sFC2Q3}, the second condition  is satisfied in every step. So the proposition follows by induction.
\end{proof}

\begin{proposition}\label{prop:deg7}
Suppose $m_2+m_3\leq n+1$. Then $\mathrm{P}_{n,7}(2^{(m_2)},\ldots,7^{(m_7)})$ holds.
\end{proposition}
\begin{proof}
With the same argument as in Proposition \ref{prop:deg4}, it is enough to prove the case $m_6 = m_7=~0$. We prove that the property $\mathrm{P}_{n,7}(2^{(m_2)},3^{(m_3)},4^{(m_4)},5^{(m_5)})$ holds using induction on $m_4$ and $m_5$. Note that $\mathrm{P}_{n,7}(2^{(m_2)},3^{(m_3)})$ holds by Lemma~\ref{lemma:P4lows}, $\mathrm{P}_{n,7}(2^{(m_2)},3^{(m_3)},4^{(m_4)})$ follows from $\mathrm{P}_{n,7}(2^{(m_2)},3^{(m_3)},4^{(m_4-1)})$ and $\mathrm{Q}_{n,3}(2^{(m_2)},3^{(m_3)},4^{(m_4-1)})$ for $m_4>0$ by Lemma~\ref{lemma:wFC} and that $\mathrm{P}_{n,7}(2^{(m_2)},3^{(m_3)},4^{(m_4)},5^{(m_5)})$ follows from $\mathrm{P}_{n,7}(2^{(m_2)},3^{(m_3)},4^{(m_4)},5^{(m_5-1)})$ and $\mathrm{Q}_{n,2}(2^{(m_2)},3^{(m_3)},4^{(m_4)},5^{(m_5-1)})$ for $m_5>0$ again by Lemma~\ref{lemma:wFC}. By Lemma~\ref{lemma:P4lows}, the second condition is satisfied in every step. So the proposition follows by induction.
\end{proof}

\begin{proposition}\label{prop:deg9}
Suppose that $m_2+m_3+m_4\leq n+1$. Then $\mathrm{P}_{n,9}(2^{(m_2)},\ldots,9^{(m_9)})$ holds.
\end{proposition}
\begin{proof}
With the same argument as in Proposition \ref{prop:deg4}, it is enough to prove the case $m_8 = m_9 = 0$. We prove that $\mathrm{P}_{n,9}(2^{(m_2)},\ldots,7^{(m_7)})$ holds using induction on $m_4$, $m_5$ and $m_6$. Note that 
\[
\mathrm{P}_{n,9}(2^{(m_2)},3^{(m_3)},4^{(m_4)})
\]
holds by Lemma~\ref{lemma:P4lows} and that $\mathrm{P}_{n,9}(2^{(m_2)},\ldots,5^{(m_5)})$ follows from 
\[
\mathrm{P}_{n,9}(2^{(m_2)},3^{(m_3)},4^{(m_4)},5^{(m_5-1)})\mbox{ and }\mathrm{Q}_{n,4}(2^{(m_2)},3^{(m_3)},4^{(m_4)},5^{(m_5-1)})
\]
for $m_5>0$, $\mathrm{P}_{n,9}(2^{(m_2)},\ldots,6^{(m_6)})$ follows from 
\[
\mathrm{P}_{n,9}(2^{(m_2)},\ldots,5^{(m_5)},6^{(m_6-1)})\mbox{ and }\mathrm{Q}_{n,3}(2^{(m_2)},\ldots,5^{(m_5)},6^{(m_6-1)})
\]
for $m_6>0$ and $\mathrm{P}_{n,9}(2^{(m_2)},\ldots,7^{(m_5)})$ follows from 
\[
\mathrm{P}_{n,9}(2^{(m_2)},\ldots,6^{(m_5)},7^{(m_7-1)})\mbox{ and }\mathrm{Q}_{n,2}(2^{(m_2)},\ldots,6^{(m_6)},7^{(m_7-1)})
\]
for $m_7>0$ by Lemma~\ref{lemma:wFC}. By Lemma~\ref{lemma:P4lows}, the second condition  is satisfied in every step. So the proposition follows by induction.
\end{proof}

\begin{remark}\label{rmk:d=8fail_1}
Under the assumption $m_2+m_3+m_4 \leq n+1$, we also have $P_{n,8}(2^{(m_2)},\ldots,8^{(m_8)})$. We do not state this here explicitly because it would not be enough to prove our main result (Theorem~\ref{thm:main}) for $d = 8$; see Remark \ref{rmk:d=8fail_2}.\remarkend
\end{remark}

\begin{proof}[Proof of Theorem~\ref{thm:main}]
For $d \leq 3$ there is nothing to prove. So we let $d \in \{4,5,6,7,9\}$.

Since by definition $\sigma_{\slrk_{n,d}^\circ}(X_{(1,d-1)}) = \bbP(\calS_d)$, it is enough to prove that, for any $(\ell_1,\ldots,\ell_{\lfloor d/2 \rfloor})$ such that $r = \ell_1+\ldots+\ell_{\lfloor d/2\rfloor} < \slrk_{n,d}^\circ$, we have
\begin{equation}\label{ineq:thm_dimEquality}
\dim J_{\ell_1,\ldots,\ell_{\left\lfloor d/2 \right\rfloor}} \leq \dim \sigma_r(X_{(1,d-1)}).
\end{equation}
Recall that the general tangent space to the join variety $J_{\ell_1,\ldots,\ell_{\left\lfloor d/2 \right\rfloor}}$ corresponds to the projectivization of an ideal generated by $2r$ general forms of degrees $(1^{(m_1)},2^{(m_2)},\ldots)$ where:
\begin{align*}
	\text{ for } d = 4:~ & \quad m_1 = m_3 = \ell_1, \quad m_2 = 2\ell_2, \quad m_i = 0 \text{ for }i\geq 4; \\
	\text{ for } d = 5:~ & \quad m_1 = m_4 = \ell_1, \quad m_2 = m_3 = \ell_2, \quad m_i = 0 \text{ for }i\geq 5; \\
	\text{ for } d = 6:~ & \quad m_1 = m_5 = \ell_1, \quad m_2 = m_4 = \ell_2, \quad m_3 = 2\ell_3; \quad m_i = 0 \text{ for }i\geq 6; \\
	\text{ for } d = 7:~ & \quad m_1 = m_6 = \ell_1, \quad m_2 = m_5 = \ell_2, \quad m_3 = m_4 = \ell_3, \quad m_i = 0 \text{ for }i\geq 7; \mbox{ and}\\
	\text{ for } d = 9:~ & \quad m_1 = m_8 = \ell_1, \quad m_2 = m_7 = \ell_2, \quad m_3 = m_6 = \ell_3,\quad m_4 = m_5 = \ell_4, \quad m_i = 0 \text{ for }i\geq 9.
\end{align*}
Since $d \leq 7$ or $d=9$ and $\ell_1 + \ldots + \ell_{\lfloor d/2 \rfloor} < \slrk_{n,d}^\circ \leq n$, by Propositions~\ref{prop:deg4}, \ref{prop:deg5}, \ref{prop:deg6}, \ref{prop:deg7} and \ref{prop:deg9}, we are in a setting where the property $\mathrm{P}$ holds. In particular, we have 
\begin{equation}
\codim J_{\ell_1,\ldots,\ell_{\left\lfloor d/2 \right\rfloor}} \geq f_{n,d}(m(\ell_1),\ell_2,\ldots,\ell_{\lfloor d/2\rfloor}).
\end{equation}
As explained in Remark \ref{rmk:HL}, we have 
\[
\codim \sigma_r(X_{(1,d-1)}) = f_{n,d}(m(r),0,\ldots,0).
\]
Hence, \eqref{ineq:thm_dimEquality} follows by Lemma \ref{lemma:arithmeticLemma} and this concludes the proof.
\end{proof}

\begin{remark}\label{rmk:d=8fail_2}
In degree $d = 8$, we would consider an ideal generated by $2r$ general forms of degrees $(1^{(m_1)},2^{(m_2)},\ldots,7^{(m_7)})$ where
\[
m_1 = m_7 = \ell_1, \quad m_2 = m_6 = \ell_2, \quad m_3 = m_5 = \ell_3, \quad m_4 = 2\ell_4.
\]
Therefore, the condition $\ell_1 + \ldots + \ell_4 \leq n$, which is the one we assume in the proof of Theorem \ref{thm:main}, is not enough to guarantee that $m_2 + m_3 + m_4 \leq n+2- m_1$, which is the condition under which we can prove the property $\mathrm{P}$ as in the previous section; see Remark \ref{rmk:d=8fail_1}. \remarkend
\end{remark}

{
\begin{remark}\label{rmk:one-maximal}
From the proof of Theorem~\ref{thm:main} and Lemma~\ref{lemma:arithmeticLemma}, we see that for $d\leq 7$ and $d=9$, the variety $\sigma_r(X_{(1,d-1)})$ is in fact the unique maximal-dimensional component  $\sigma_r(X_{\rm red})$ for $(n,d,r)\neq (3,4,2)$. {Indeed, for $n=3$ and $d = 4$ we have that
\[
	\codim \sigma_2(X_{(1,3)}) = \codim J(X_{1,3},X_{2,2}) = \codim \sigma_2(X_{(2,2)}) = 1.
\]}
\remarkend
\end{remark}
}

\begin{remark}
As already noticed in \cite{C+19:SecantReducible}, from \eqref{ineq:dimJoins} we may observe that the varieties $J_{\ell_1,\ldots,\ell_{\left\lfloor d/2 \right\rfloor}}$ are highly \textit{defective}, i.e., their dimensions are strictly smaller than the one expected by a direct count of parameters. Indeed we see that they are defective as soon as $\sum_{i \in \alpha} d_i < d$ for some $\alpha$ with $|\alpha| \geq 2$. This is due to the presence of Koszul syzygies in degree $d$ for a general ideal corresponding to a general tangent space at such join variety, as considered in the proof of Theorem \ref{thm:main}.\remarkend
\end{remark}

\section{Proof of the key inequality}\label{sec:proofkeylemma}

Fix integers $n\geq 1$ and $d\geq 4$ and write $m(\ell):=n-\ell$. For integers $m,\ell_2,\ldots,\ell_{\left\lfloor d/2 \right\rfloor}\geq0$, let
\[
f_{n,d}(m,\ell_2,\ldots,\ell_{\left\lfloor d/2 \right\rfloor}) := \sum_{\beta\in\bbZ_{\geq0}^{2\times(\lfloor d/2\rfloor-1)}}(-1)^{|\beta|}\prod_{i=2}^{\lfloor d/2\rfloor}\binom{\ell_i}{\beta_{1,i}}\binom{\ell_i}{\beta_{2,i}}\cdot\binom{m+d-||\beta||}{m} - (n-m)(m+1),
\]
where we write $||\beta||:=\sum_{i=2}^{\lfloor d/2\rfloor}\left(\beta_{1,i}\cdot i+\beta_{2,i}\cdot(d-i)\right)$ and $|\beta|:=\sum_{i=2}^{\lfloor d/2\rfloor}\left(\beta_{1,i}+\beta_{2,i}\right)$ for all
\[
\beta=\begin{pmatrix}\beta_{1,2}&\beta_{1,3}&\cdots&\beta_{1,\lfloor d/2\rfloor}\\\beta_{2,2}&\beta_{2,3}&\cdots&\beta_{2,\lfloor d/2\rfloor}\end{pmatrix}\in\bbZ_{\geq0}^{2\times(\lfloor d/2\rfloor-1)}.
\]
Our goal is to prove that 
\begin{equation}\label{eq:lastsecgoal}
f_{n,d}(m(\ell_1),\ell_2,\ldots,\ell_{\lfloor d/2\rfloor}) \geq f_{n,d}(m(\ell_1 + \ldots + \ell_{\lfloor d/2\rfloor}),0,\ldots,0)
\end{equation}
when $\ell_1 +\ldots + \ell_{\lfloor d/2\rfloor} \leq \slrk_{n,d}^\circ - 1$. Note that this is equivalent to proving that
\[
f_{n,d}(m,\ell_2,\ldots,\ell_{\lfloor d/2\rfloor}) \geq f_{n,d}(m-(\ell_2 + \ldots + \ell_{\lfloor d/2\rfloor}),0,\ldots,0)
\]
when $m\leq n$ and $m-(\ell_2 +\ldots + \ell_{\lfloor d/2\rfloor}) \geq n-\slrk_{n,d}^\circ + 1$. 

We consider the following statements: for integers $3\leq j\leq \lfloor d/2\rfloor$, $m\leq n$, $\ell_2,\ldots,\ell_{j-1}\geq0$ and $\ell_j>0$ such that 
\[
m-(\ell_2 +\ldots + \ell_j) \geq n-\slrk_{n,d}^\circ + 1
\]
holds, we denote by $\calA^{(j)}_{n,d}(m,\ell_2,\ldots,\ell_j)$ that the inequality
\[
f_{n,d}(m,\ell_2,\ldots,\ell_{j-2},\ell_{j-1},\ell_j,0,\ldots,0){>} f_{n,d}(m,\ell_2,\ldots,\ell_{j-2},\ell_{j-1}+1,\ell_j-1,0,\ldots,0)
\]
holds and, for integers $m\leq n$ and $\ell>0$ such that $m-\ell\geq n-\slrk_{n,d}^\circ + 1$ holds, we denote by $\calB_{n,d}(m,\ell)$ that the inequality
\[
f_{n,d}(m,\ell,0,\ldots,0){>} f_{n,d}(m-1,\ell-1,0,\ldots,0)
\]
holds. Clearly, the statements $\calA^{(j)}_{n,d}(m,\ell_2,\ldots,\ell_j)$ and the statements $\calB_{n,d}(m,\ell)$ together imply (\ref{eq:lastsecgoal}) {holds strictly} (for the integers $n,d$ we fixed). We will prove that these statements hold $n\gg0$. 

The proof of Lemma~\ref{lemma:arithmeticLemma} is divided into two parts: first,
\begin{itemize}
\item for $d=4$ take $N:=755$;
\item for $d=5$ take $N:=3056$;
\item for $d=6$ take $N:=1742$;
\item for $d=7$ take $N:=32215$;
\item for $d=8$ take $N:=1408841$;
\item for $d=9$ take $N:=73305293$; 
\item for $d=10$ take $N:=4393224603$; 
\end{itemize}
and take $N\gg0$ when $d\geq 11$.

\begin{remark}
These $N$ have been picked to satisfy certain properties. In particular, they are higher than the highest real roots of certain polynomials. In such cases, these real roots have been approximated using mathematical computer software.
\remarkend
\end{remark}

\begin{lemma}\label{lemma:final1}
The statements $\calA^{(j)}_{n,d}(m,\ell_2,\ldots,\ell_j)$ and the statements $\calB_{n,d}(m,\ell)$ hold for all $n>N$. 
\end{lemma}

\begin{lemma}\label{lemma:final2}
Assume that $d\leq 10$ and $n\leq N$. Then 
\[
f_{n,d}(m,\ell_2,\ldots,\ell_{\lfloor d/2\rfloor}) \geq f_{n,d}(m-(\ell_2 + \ldots + \ell_{\lfloor d/2\rfloor}),0,\ldots,0)
\]
for all integers integers $m,\ell_2,\ldots,\ell_{\left\lfloor d/2 \right\rfloor}\geq0$ such that $m-(\ell_2 +\ldots + \ell_{\lfloor d/2\rfloor}) \geq n-\slrk_{n,d}^\circ + 1$. {If $\ell_2+\ldots+\ell_{\lfloor d/2\rfloor}>0$ and $(n,d,m-\ell_2)\neq (3,4,1)$, then the inequality holds strictly.}
\end{lemma}

Together, these lemmas imply Lemma~\ref{lemma:arithmeticLemma}. Our goals now are to prove these lemmas. We start by giving some lower and upper bounds for $n-\slrk^\circ_{n,d}$. Recall that
\[
\slrk^\circ_{n,d} = \min\left\{r \in \bbZ_{\geq0} \,\middle|\, r(n+1-r) \geq \binom{d + n - r}{d}\right\}
\]
for $d\geq3$.

\begin{lemma}\label{lemma:appendix1}
Take $p(x):=(x+d)\cdots(x+2)-d!(n-x)$.
\begin{enumerate}
\item The polynomial $p(x)$ has a unique positive root $a>0$.
\item We have $a<\sqrt[d-1]{d!n}-2$ and $n-\slrk^{\circ}_{n,d}=\lfloor a\rfloor$.
\item Suppose that $d\geq 4$ and that
\[
n\geq\frac{1}{d!}\max\left\{d^{d-1},\left((d-1)!\right)^{\frac{d-1}{d-3}}\right\}.
\]
Then $a>\sqrt[d-1]{d!n}-(d+2)$.
\item Suppose that $d\geq5$ and that
\[
n\geq \frac{1}{d!}\max\left\{d^{d-1},\left(\frac{d!}{\left(d/2-1\right)^2}\right)^{\frac{d-1}{d-4}}\right\}.
\]
Then $a> \sqrt[d-1]{d!n}-\left(d/2 +1\right)$.
\end{enumerate}  
\end{lemma}
\begin{proof}
Take $x=n-r$. Then we see that 
\[
r(n+1-r) \geq \binom{d + n - r}{d}
\]
holds if and only if $p(x)\leq 0$. Note that $p(0)=d!-d!n\leq0$ and that $p(x)$ is strictly increasing on $\bbR_{\geq0}$. So $p(x)$ has a unique positive root $a>0$. So $\lfloor a\rfloor$ is the maximal integer such that $p(x)\leq n$ and hence $\lfloor a\rfloor=n-\slrk^{\circ}_{n,d}$. Take $x=\sqrt[d-1]{d!n}-2>0$. Then
\[
p(x)\geq d!n-d!(n-x)=d!x>0
\]
and so $a<\sqrt[d-1]{d!n}-2$. 

Assume that the conditions of (3) holds and take $y=\sqrt[d-1]{d!n}\geq d$. Then
\[
p(y-(d+2))<y^{d-2}(y-d)-d!(n-y)=d!y-dy^{d-2}=dy((d-1)!-y^{d-3})\leq 0
\]
since $(d!n)^{d-3}\geq (d-1)!^{d-1}$ and hence $(d-1)!\leq y^{d-3}$. So $a>\sqrt[d-1]{d!n}-(d+2)$. 

Assume that the conditions of (4) holds and again take $y=\sqrt[d-1]{d!n}\geq d$. Then
\[
p\left(y-\left(d/2+1\right)\right)<\left(y+d/2-1\right)\cdots\left(y-\left(d/2-1\right)\right)-d!(n-y).
\]
Note that
\[
\left(y+d/2-1\right)\cdots\left(y-\left(d/2-1\right)\right)=\begin{cases}
y\prod_{i=1}^{d/2-1}(y^2-\left(d/2-i\right)^2)&\mbox{if $d$ is even}\\
\prod_{i=1}^{\lfloor d/2\rfloor}(y^2-\left(d/2-i\right)^2)&\mbox{if $d$ is odd}
\end{cases}
\]
and so this product is at most $y^{d-3}\left(y^2-\left(d/2-1\right)^2\right)$. Hence
\[
p\left(y-\left(d/2-1\right)\right)<y^{d-3}\left(y^2-\left(d/2-1\right)^2\right)-d!(n-y)=y\left(d!-\left(d/2-1\right)^2y^{d-3}\right)\leq 0
\]
since $d!\leq\left(d/2-1\right)^2y^{d-3}$. So $a> \sqrt[d-1]{d!n}-\left(d/2 +1\right)$. 
\end{proof}

\subsection*{Proof of the first lemma}

Assume that $n>N$. If $d\leq 5$, then we have chosen $N$ such that the condition of Lemma~\ref{lemma:appendix1}(3) holds and we take $w:=d+2$. Otherwise, we have chosen $N$ such that the condition of Lemma~\ref{lemma:appendix1}(4) holds and we take $w:=d/2+1$. We first prove the statements $\calA^{(j)}_{n,d}(m,\ell_2,\ldots,\ell_j)$.

Let $3\leq j\leq \lfloor d/2\rfloor$, $m\leq n$, $\ell_2,\ldots,\ell_{j-1}\geq0$ and $\ell_j>0$ be integers such that 
\[
m-(\ell_2 +\ldots + \ell_j) \geq n-\slrk_{n,d}^\circ + 1
\]
holds and take
\[
g^{(j)}_{n,d}(m,\ell_2,\ldots,\ell_j):=f_{n,d}(m,\ell_2,\ldots,\ell_j,0,\ldots,0)- f_{n,d}(m,\ell_2,\ldots,\ell_{j-2},\ell_{j-1}+1,\ell_j-1,0,\ldots,0){-1};
\]
view $g^{(j)}_{n,d}$ as a polynomial in $m,\ell_2,\ldots,\ell_j$. Note that it has degree $d-(j-1)$ and that its homogeneous part of top degree equals $m^{d-(j-1)}/(d-(j-1))!$. 

\begin{lemma}\label{lemma:appendix2}
Write $g^{(j)}_{n,d}=\sum_{i,\alpha}c_{i,\alpha}m^i\ell^\alpha$. Take
\[
\tilde{c}_{i,\alpha}=\begin{cases}0&\mbox{if $c_{i,\alpha}>0$ and $\alpha\neq0$}\\
c_{i,\alpha}&\mbox{otherwise}\\
\end{cases}
\]
and $\tilde{g}^{(j)}_{n,d}=\sum_{i,\alpha}\tilde{c}_{i,\alpha}m^{i+\alpha_2+\ldots+\alpha_j}$. Then 
\[
g^{(j)}_{n,d}(m,\ell_2,\ldots,\ell_j)\geq \tilde{g}^{(j)}_{n,d}(m)
\]
for all $0\leq\ell_2,\ldots,\ell_j\leq m$.
\end{lemma}
\begin{proof}
This follows from the fact that $c_{i,\alpha}m^i\ell^\alpha\geq\tilde{c}_{i,\alpha}m^{i+\alpha_2+\ldots+\alpha_j}$ for all $i,\alpha$.
\end{proof}

We have $m-(\ell_2 +\ldots + \ell_{\left\lfloor d/2 \right\rfloor}) \geq n-\slrk^\circ_{n,d}+1\geq 1$. Hence, by Lemma \ref{lemma:appendix2}, to prove that $g^{(j)}_{n,d}$ is positive, it suffices to prove that $\tilde{g}^{(j)}_{n,d}$ is positive. Note that 
\[
\tilde{g}^{(j)}_{n,d}(x) := \sum_{i,\alpha}\widetilde{c}_{i,\alpha}x^{i+\alpha_2+\ldots+\alpha_j}
\]
has degree $d-j+1$ and that its top coefficient equals $1/(d-j+1)!>0$. So, $\tilde{g}^{(j)}_{n,d}(x)\rightarrow\infty$ as $x\rightarrow\infty$. In particular, we deduce that $g^{(j)}_{n,d}$ is positive whenever $m$ is bigger than the biggest real root $x^{(j)}_{*}$ of $\tilde{g}^{(j)}_{n,d}(x)$. 
Let $a$ be as in Lemma~\ref{lemma:appendix1}.
Since $\ell_j>0$ and $m-(\ell_2 +\ldots + \ell_j) \geq n-\slrk_{n,d}^\circ + 1$, we have 
\[
m\geq  n-\slrk^{\circ}_{n,d}+2=\lfloor a\rfloor + 2\geq \sqrt[d-1]{d!n}-w+1.
\]
We have chosen $N$ such that $N\geq (x^{(j)}_{*}+w-1)^{d-1}/d!$. So since $n>N$, it follows that $m \geq x^{(j)}_{*}$. So the statements $\calA^{(j)}_{n,d}(m,\ell_2,\ldots,\ell_j)$ hold.

\begin{example}
For $d=6$ and $j=3$, we see that $f_{n,d}(m,\ell,\ell')+(n-m)(m+1)$ equals
\[
\binom{m+6}{6}-\ell\binom{m+2}{2}-\ell\binom{m+4}{4}-2\ell'\binom{m+3}{3}+\binom{\ell}{2}\binom{m+2}{2}-\binom{\ell}{3}+2\ell\ell'+\ell^2+\binom{2\ell'}{2}
\]
and
\begin{align*}
g^{(j)}_{n,d}(m,\ell,\ell')&:=f_{n,d}(m,\ell,\ell')-f_{n,d}(m,\ell+1,\ell'-1){-1}\\
&=\frac{1}{24}m^4 + \frac{1}{12}m^3 - \frac{1}{2}m^2\ell - \frac{1}{24}m^2 - \frac{3}{2}m\ell + \frac{1}{2}\ell^2 - \frac{1}{12}m - \frac{3}{2}\ell + 2\ell' {- 3}.
\end{align*}
So we get
\[
\tilde{g}^{(j)}_{n,d}(m)=\frac{1}{24}m^4 + \frac{1}{12}m^3 - \frac{1}{2}m^3 - \frac{1}{24}m^2 - \frac{3}{2}m^2 + 0 - \frac{1}{12}m - \frac{3}{2}m + 0 {- 3}.
\]
One can verify numerically that this polynomial has highest root $13.0 < x^{(3)}_{*}<13.1$.\exampleend
\end{example}

Next, we prove the statements $\calB_{n,d}(m,\ell)$. 

Let $m\leq n$ and $\ell>0$ be integers such that $m-\ell\geq n-\slrk_{n,d}^\circ + 1$ holds and write
\[
f_{n,d}(m,\ell,0,\ldots,0)-f_{n,d}(m-1,\ell-1,0\ldots,0) -1 =g_{n,d}(m,\ell)-n;
\]
view $g_{n,d}$ as a polynomial in $m,\ell$. Note that $n\leq (m+w-1)^{d-1}/d!$ and again construct $\tilde{g}_{n,d}$ from~$g_{n,d}$ as in Lemma~\ref{lemma:appendix2}. Then we see that
\[
f_{n,d}(m,\ell,0,\ldots,0) > f_{n,d}(m-1,\ell-1,0\ldots,0)
\] 
holds when $\tilde{g}_{n,d}(m)- (m+w-1)^{d-1}/d!\geq0$. Both $\tilde{g}_{n,d}(x)$ and $(x+w-1)^{d-1}/d!$ have degree $d-1$ in $x$. The leading coefficient of $\tilde{g}_{n,d}$ equals $1/(d-1)!$ and the leading coefficient of $(m+w-1)^{d-1}/d!$ equals $1/d!$. It follows that $\tilde{g}_{n,d}(x)- (x+w-1)^{d-1}/d!\to\infty$ as $x\to\infty$. Similar to before, we have chosen $N$ such that the statements $\calB_{n,d}(m,\ell)$ hold. More precisely, we have $N\geq (x^*+w-1)^{d-1}/d!$ where $x_*$ is the biggest real root of $\tilde{g}_{n,d}(x)- (x+w-1)^{d-1}/d!$. This finishes the proof of Lemma~\ref{lemma:final1}.

\begin{example}
For $d=4$, we have
\[
f_{n,d}(m,\ell) = \binom{m+4}{4} - 2\ell\binom{m+2}{2}+\binom{2\ell}{2}-(n-m)(m+1) 
\]
and
\begin{align*}
g_{n,d}(m,\ell)&:=f_{n,d}(m,\ell)-f_{n,d}(m-1,\ell-1) -1+n\\
&=\frac{1}{6}m^3 - 2m\ell + \frac{17}{6}m + 2\ell - {3}.
\end{align*}
So we get
\[
\tilde{g}_{n,d}(x)-(x+w-1)^{d-1}/d!=\left(\frac{1}{6}x^3 - 2x^2 + \frac{17}{6}x+0 - {3}\right)-\frac{(x+5)^{3}}{24}.
\]
One can verify numerically that this polynomial has highest root $21.2 < x_*<21.3$.\exampleend
\end{example}

\subsection*{Proof of the second lemma}

Assume that $d\leq 10$. Note that, in principle, {we can verify that Lemma~\ref{lemma:final2} holds} with the support of algebra software in finite time. Below, we study the inequalities that we need to check in more detail to make this finite time more manageable.

For all integers $n\geq1$, we need to prove that
\begin{equation}\label{eq:blabla}
f_{n,d}(m,\ell_2,\ldots,\ell_{\lfloor d/2\rfloor}) \geq f_{n,d}(m-(\ell_2 + \ldots + \ell_{\lfloor d/2\rfloor}),0,\ldots,0)
\end{equation}
for all integers $m\leq n$ and $\ell_2,\ldots,\ell_{\lfloor d/2\rfloor}\geq0$ with $m-(\ell_2 +\ldots + \ell_{\lfloor d/2\rfloor}) \geq n-\slrk_{n,d}^\circ + 1$. {We also need to show that this inequality holds strictly when $\ell_2+\ldots+\ell_{\lfloor d/2\rfloor}>0$ and $(n,d,m-\ell_2)\neq(3,4,1)$.} We already know that the inequality holds {strictly} when $n>N$. The following lemma shows that for fixed $m,\ell_2,\ldots,\ell_{\lfloor d/2\rfloor}$ we only need to check this inequality for the highest $n$ such that $m-(\ell_2 +\ldots + \ell_{\lfloor d/2\rfloor}) \geq n-\slrk_{n,d}^\circ + 1$.

\begin{remark}\label{rmk:non-decreasing}
Note that $\slrk_{n,d}^\circ\leq\slrk_{n+1,d}^\circ\leq\slrk_{n,d}^\circ+1$ for every $n\geq 1$ since every polynomial in $n+1$ variables is also a polynomial in $n+2$ variables and every polynomial in $n+2$ variables can be written as the sum of a polynomial in $n+1$ variables and a multiple the remaining variable. In particular, the expression $n-\slrk_{n,d}^\circ + 1$ is a non-decreasing function of $n$.\remarkend
\end{remark}

\begin{lemma}
Let $m\leq n$ and $\ell_2,\ldots,\ell_{\lfloor d/2\rfloor}\geq0$ be integers with 
\[
m-(\ell_2 +\ldots + \ell_{\lfloor d/2\rfloor}) \geq (n+1)-\slrk_{n+1,d}^\circ + 1
\]
and suppose that
\[
f_{n+1,d}(m,\ell_2,\ldots,\ell_{\lfloor d/2\rfloor}) \geq f_{n+1,d}(m-(\ell_2 + \ldots + \ell_{\lfloor d/2\rfloor}),0,\ldots,0)
\]
holds. Then (\ref{eq:blabla}) holds as well. {If in addition $\ell_2+\ldots+\ell_{\lfloor d/2\rfloor}>0$, then (\ref{eq:blabla}) holds strictly.}
\end{lemma}
\begin{proof}
Fix $m,\ell_2,\ldots,\ell_{\lfloor d/2\rfloor}$ and view 
\[
f_{n,d}(m,\ell_2,\ldots,\ell_{\lfloor d/2\rfloor}) - f_{n,d}(m-(\ell_2 + \ldots + \ell_{\lfloor d/2\rfloor}),0,\ldots,0).
\]
as a function of $n$. Note that
\begin{align*}
f_{n,d}(m,\ell_2,\ldots,\ell_{\lfloor d/2\rfloor})&= c - n(m+1),\\
f_{n,d}(m-(\ell_2 + \ldots + \ell_{\lfloor d/2\rfloor}),0,\ldots,0)&={c'}-n(m-(\ell_2 + \ldots + \ell_{\lfloor d/2\rfloor})+1)
\end{align*}
for some constants $c,{c'}$ (depending on only $m,\ell_2,\ldots,\ell_{\lfloor d/2\rfloor}$) and have difference $c-{c'}-n(\ell_2 + \ldots + \ell_{\lfloor d/2\rfloor})$. By assumption, we have $c-{c'}-(n+1)(\ell_2 + \ldots + \ell_{\lfloor d/2\rfloor})\geq0$ and hence $c-{c'}-n(\ell_2 + \ldots + \ell_{\lfloor d/2\rfloor})\geq0$ holds as well since $\ell_2,\ldots,\ell_{\lfloor d/2\rfloor}\geq0$. {If $\ell_2+\ldots+\ell_{\lfloor d/2\rfloor}>0$, then the latter holds strictly.}
\end{proof}

Using induction from $n=N+1$ going down, the lemma shows that it suffices to check that (\ref{eq:blabla}) holds for all integers $m\leq n$ and $\ell_2,\ldots,\ell_{\lfloor d/2\rfloor}\geq0$ with 
\[
n-\slrk_{n,d}^\circ + 1\leq m-(\ell_2 +\ldots + \ell_{\lfloor d/2\rfloor})<(n+1)-\slrk_{n+1,d}^\circ + 1 .
\]
By Remark \ref{rmk:non-decreasing}, we see that 
\[
n-\slrk_{n,d}^\circ + 1<(n+1)-\slrk_{n+1,d}^\circ + 1 .
\]
if and only if $\slrk_{n,d}^\circ=\slrk_{n+1,d}^\circ$. Hence, we are left with the following cases.

\begin{claim}\label{claim:last}
Let $n\leq N$ be an integer such that $\slrk_{n,d}^\circ=\slrk_{n+1,d}^\circ$ and let $m\leq n$ and $\ell_2,\ldots,\ell_{\lfloor d/2\rfloor}\geq0$ be integers with $m-(\ell_2 +\ldots + \ell_{\lfloor d/2\rfloor}) = n-\slrk_{n,d}^\circ + 1$. Then 
\[
f_{n,d}(m,\ell_2,\ldots,\ell_{\lfloor d/2\rfloor}) \geq f_{n,d}(n-\slrk_{n,d}^\circ + 1,0,\ldots,0)
\]
holds. {If $\ell_2+\ldots+\ell_{\lfloor d/2\rfloor}>0$ and $(n,d)\neq (3,4)$, then the inequality holds strictly.}
\end{claim}

\begin{remark}
By Lemma~\ref{lemma:appendix1}, we know that $\slrk_{n,d}^\circ\approx n-\sqrt[d-1]{d!n}$. By Remark \ref{rmk:non-decreasing},
it follows that the number of $n\leq N$ such that $\slrk_{n,d}^\circ=\slrk_{n+1,d}^\circ$ is around $\sqrt[d-1]{d!N}\ll N$. In particular, it is not efficient to check the condition $\slrk_{n,d}^\circ=\slrk_{n+1,d}^\circ$ for every $n\leq N$.\remarkend
\end{remark}

As in the prove of the first lemma, we know that the statements $\calA^{(j)}_{n,d}(m,\ell_2,\ldots,\ell_j)$ and the statements $\calB_{n,d}(m,\ell)$ hold when $m\geq\max(x^{(3)}_{*},\ldots,x^{(\lfloor d/2\rfloor)}_{*},x_*)$. So the cases where this holds reduce to the cases where $m\leq\lfloor\max(x^{(3)}_{*},\ldots,x^{(\lfloor d/2\rfloor)}_{*},x_*)\rfloor$. Note that the replacements of $(m,\ell_1,\ldots,\ell_{\lfloor d/2\rfloor})$ that occur here do not change $m-(\ell_2 +\ldots + \ell_{\lfloor d/2\rfloor})$. Under this addition condition, we checked the claim for $d \leq 10$ using a combination of \verb|SAGE| and \verb|NumPy|. The files containing the code used are available as ancillary files of the arXiv version and on the personal webpage of the first author. The running time required was less than two minutes on a laptop. We found that the claim holds and that concludes the proof of Lemma \ref{lemma:final2}.
\newcommand{\etalchar}[1]{$^{#1}$}

\end{document}